\RequirePackage{fix-cm}

\documentclass[smallextended]{svjour3}       

\smartqed  
\usepackage{graphicx}
\usepackage{amssymb,amsfonts,amsmath}

\usepackage{ulem}
\usepackage{xcolor}

\begin{document}

\title{Subdifferential characterization of probability functions under Gaussian distribution\thanks{This work is partially supported by CONICYT grants: Fondecyt 1151003, Fondecyt 1150909, Basal PFB-03 and Basal FB0003, CONICYT-PCHA/doctorado Nacional/2014-21140621.  The second author thanks the Deutsche Forschungsgemeinschaft for their support within project B04 in the Sonderforschungsbereich / Transregio 154 Mathematical Modelling, Simulation and Optimization using the Example of Gas Networks.}
}

\titlerunning{Subdifferentials of probability functions}        

\author{Abderrahim Hantoute         \and
	Ren{\'e} Henrion \and
	Pedro P{\'e}rez-Aros 
}

\authorrunning{A. Hantoute, R. Henrion and P. P{\'e}rez-Aros} 

\institute{A. Hantoute  \at
	Center for Mathematical Modeling,
	Universidad de Chile, Santiago, Chile\\
	\email{ahantoute@dim.uchile.cl }
\and
R. Henrion \at
Weierstrass Institute for Applied Analysis and Stochastics, Berlin, Germany\\
\email{henrion@wias-berlin.de}
\and
P. P{\'e}rez-Aros \at
Department of Mathematical Engineering, Universidad de Chile, Santiago, Chile\\
\email{pperez@dim.uchile.cl}
}
	

\date{\bfseries{\today}}

\maketitle

\begin{abstract}
Probability functions figure prominently in optimization problems of engineering. They may be nonsmooth even if all input
data are smooth. This fact motivates the consideration of subdifferentials for such typically just continuous functions.  The aim of this paper is to provide
subdifferential formulae of such functions in the case of Gaussian distributions for possibly infinite-dimensional decision variables and nonsmooth (locally Lipschitzian) input data.
These formulae are based on the spheric-radial decomposition of Gaussian random vectors on the one hand and on a cone of directions of moderate growth on the
other. By  successively adding additional hypotheses, conditions are satisfied under which the probability function is locally Lipschitzian or even differentiable.

\keywords{probability functions, probabilistic constraint, stochastic optimization, multivariate
	Gaussian distribution, spheric-radial decomposition, Clarke subdifferential, Mordukhovich subdifferential}
\subclass{90C15 \and  90C30 \and 49J52 \and 49J53}
\end{abstract}

\section{Introduction}
\noindent The aim of this paper is to investigate subdifferential properties
of Gaussian probability functions induced by nonnecessarily smooth initial data. This topic
combines aspects of stochastic programming with arguments from variational
analysis, two areas which have been crucially influenced by the fundamental
work of Prof. Roger J-B Wets (see, e.g., \cite{rockwets}, \cite{wets} and
many other references). The motivation to study analytical properties of
probability functions comes from their importance in the context of
engineering problems affected by random parameters. They are at the core of
probabilistic programming (i.e., optimization problems subject to
probabilistic constraints) (e.g., \cite{prek}, \cite{shapdenrusz}) or of
reliability maximization (e.g., \cite{dit}). 

A probability function assigns to a control or decision variable the
probability that a certain random inequality system induced by this decision
variable be satisfied (see (\ref{probcons}) below). Since such functions are
typical constituents of optimization problems under uncertainty, it is
natural to ask for their analytical properties, first of all
differentiability. Roughly speaking, this can be guaranteed under three
assumptions: the differentiability of the input data, an appropriate
constraint qualification for the given random inequality system and the
compactness of the set of realizations of the random vector for the fixed
decision vector (e.g., \cite{kibur}, \cite{pflug}, \cite{uryas}). While the
first two assumptions are quite natural, the last one appears to be
restrictive in problems involving random vectors with unbounded support.
Failure of the compactness condition, however, may result in general in
nonsmoothness of the probability function despite the fact that all input
data are smooth and a standard constraint qualification is satisfied (see 
\cite[Prop. 2.2]{ackooij}). In order to keep the differentiability while
doing without the compactness assumption, one may restrict to special
distributions such as Gaussian or Gaussian-like as in \cite{ackooij}, \cite%
{ackooij2}. The working horse for deriving differentiability and gradient
formulae in these cases is the so-called \textit{spheric-radial decomposition%
} of Gaussian random vectors \cite[p. 29]{genz}. The resulting formulae for
the gradient of the probability function are represented - similar to the
formulae for the probability values themselves - as integrals over the unit
sphere with respect to the uniform measure. The latter can be efficiently
approximated by QMC methods tailored to this specific measure (e.g., \cite%
{brauch}). Such approach, by exploiting special properties of the
distribution, promises more efficiency in the solution of probabilistic
programs than general gradient formulae in terms of possibly complicated
surface or volume integrals. Successful applications of this methodology in
the context of probabilistic programming in gas network optimization is
demonstrated in \cite{gotzes}, \cite{grandon}. 

The aim of this paper is to substantially extend the earlier results in \cite%
{ackooij}, \cite{ackooij2} in two directions: first, decisions will be
allowed to be infinite-dimensional and second, the random inequality may be
just locally Lipschitzian rather than smooth. As the resulting probability
function can be expected to be continuous only (rather than locally
Lipschitzian or even smooth), appropriate tools (subdifferentials) from
variational analysis will be employed for an analytic characterization. 

We consider a probability function $\varphi :X\rightarrow \mathbb{R}$
defined by 
\begin{equation}
\varphi (x):=\mathbb{P}\left( g\left( x,\xi \right) \leq 0\right),  \label{probcons}
\end{equation}%
where $X$ is a Banach space, $g$ $:X\times \mathbb{R}^{m}\rightarrow \mathbb{%
	R}$ is a function depending on the realizations of an $m$-dimensional random
vector $\xi $.
Such probability functions are important in many optimization problems
dealing with reliability maximization or probabilistic constraints. The
latter one refers to an inequality $\varphi (x)\geq p$ constraining the set
of feasible decisions in an optimization problem, in order to guarantee that
the underlying random inequality $g\left( x,\xi \right) \leq 0$ is satisfied
under decision $x$ with probability at least $p\in (0,1]$, referred to as a 
a probability level (or safety level). Since we allow in
our paper the function $g$ to be locally Lipschitzian, there is no loss of
generality in considering a single random inequality only because in a
finite system of such inequalities one could pass to the maximum of
components.

Throughout the paper, we shall make the following basic assumptions on the
data of (\ref{probcons}):
%
%
$$
\begin{array}{cl}
& 1.\ X \text{ is a reflexive and separable Banach space}.\\
& 2.\ \text{ Function } g \text{ is locally Lipschitzian as a function of both arguments} \\
(H) 	\quad &  \, \quad \text{ simultaneously, and convex as a function of the second argument.}\\
& 3.\ \text{ The random vector } \xi \text{ is Gaussian of type } \xi \sim \mathcal{N}%
	\left( 0,R\right),  \text{ where } R \\ 
& \, \quad \text{ is a correlation matrix.}	
\end{array}
$$

\noindent A brief discussion of these assumptions is in order here:
reflexivity of $X$ is imposed in order to work with the limiting
(Mordukhovich) subdifferential as introduced in Definition \ref%
{subdifferentials} below (actually, one could consider the more general case
of Asplund spaces). The separability of $X$ is needed in order to make use
of an interchange formula for the limiting subdifferential and integration
sign (see Proposition \ref{keyresult} below). For the same reason, $g$ is required to
be locally Lipschitzian. As already mentioned above, considering just one
inequality rather than a system is no more restriction then. In particular,
the single inequality $g\left( x,z\right) \leq 0$ could represent a finite
or (compactly indexed) infinite system of smooth inequalities. Considering a
Gaussian random vector $\xi $ allows one to pass to a whole class of
Gaussian-like multivariate distributions (e.g., Student, Log-normal,
truncated Gaussian, $\chi ^{2}$ etc.) upon shifting their nonlinear
transformations to a Gaussian one into a modified function $\tilde{g}$
satisfying the same assumptions as required for $g$ here (e.g. \cite[Section
4.3]{ackooij}). Moreover, assuming a centered Gaussian distribution with
unit variances isn't a restriction either, because in the general case $\xi
\sim \mathcal{N}\left( \mu ,\Sigma \right) $, we may pass to the
standardized vector $\tilde{\xi}:=D(\xi -\mu )$, where D is the diagonal
matrix with elements $D_{ii}:=1/\sqrt{\Sigma _{ii}}$. Then, as required
above, $\tilde{\xi}\sim \mathcal{N}\left( 0,R\right) $, with $R$ being the
correlation matrix associated with $\Sigma $ and so%
\begin{equation*}
	\varphi (x)=\mathbb{P}\left( g\left( x,\xi \right) \leq 0\right) =\mathbb{P}%
	\left( \tilde{g}\left( x,\tilde{\xi}\right) \leq 0\right) ;\quad \tilde{g}%
	\left( x,z\right) :=g\left( x,D^{-1}z+\mu \right) .
\end{equation*}%
Clearly, $\tilde{g}$ is locally Lipschitzian and is convex in the second
argument if $g$ is so. Hence, there is no loss of generality in assuming
that $\xi \sim \mathcal{N}\left( 0,R\right) $ from the very beginning.

Our first observation is that our basic assumptions above do not guarantee the
continuity of $\varphi $ even if $g$ is continuously differentiable. A
simple two-dimensional example is given by $g(r,s):=r\cdot s$ (which is
convex in the second argument) and $\xi \sim \mathcal{N}\left( 0,1\right) $.
Then, $\varphi (r)=0.5$ for $r\neq 0$ and $\varphi (0)=1$. Since we want to
have the continuity as a minimum initial property of $\varphi $ in our analysis,
we will add the additional assumption that $g\left( \bar{x},0\right) <0$
holds true at a point of interest $\bar{x}$ (at which a subdifferential of $%
\varphi $ is computed). In other words, given the convexity of $g$ in the
second argument, zero is a Slater point for the inequality $g\left(
x,z\right) \leq 0$, $z\in  \mathbb{R}^m$. As shown in \cite[Proposition 3.11]{ackooij}, the opposite
case would entail that $\varphi (\bar{x})\leq 0.5$. Since one deals in
typical applications like probabilistic programming or reliability
maximization with probabilities close to one, it follows that the assumption 
$g\left( \bar{x},0\right) <0$ can be made without any practical loss of
generality.

The paper is organized as follows: In Section 3 and 4, we provide all the auxiliary results (continuity and partial subdifferential of the radial probability function) which are needed for the derivation of the main subdifferential formula presented in Section 5. This main result which is valid for general continuous probability functions will be specified then by adding additional hypotheses to the locally Lipschitzian and differentiable case. An application to probability functions induced by a finite system of smooth inequalities is given in Subsection \ref{secc2}.

\section{Preliminaries}

\subsection{Spheric-radial decomposition of Gaussian random vectors}

\noindent We recall the fact that any Gaussian random vector $\xi \sim 
\mathcal{N}\left( 0,R\right) $ has a so-called spheric-radial decomposition,
which means that the probability of $\xi $ taking values in an arbitrary
Borel subset $M$ of $\mathbb{R}^{m}$ can be represented as (e.g., \cite[p.
105]{deak}) 
\begin{equation*}
	\mathbb{P}\left( \xi \in M\right) =\int\limits_{v\in \mathbb{S}^{m-1}}\mu
	_{\eta }\left( \left\{ r\geq 0\ |\ rLv\in M\right\} \right) d\mu _{\zeta }(v),
\end{equation*}%
where $\mathbb{S}^{m-1}:=\left\{ v\in \mathbb{R}^{m}\ |\ \left\Vert v\right\Vert
^{2}=1\right\} $ denotes the unit sphere in $\mathbb{R}^{m}$, $\mu _{\eta }$
is the one-dimensional Chi-distribution with $m$ degrees of freedom, and $%
\mu _{\zeta }$ refers to the uniform distribution on $\mathbb{S}^{m-1}$.
Moreover, the (non-singular) matrix $L$ is supposed to be a factor in a decomposition $%
R=LL^{T}$ of the positive definite correlation matrix $R$ (e.g. Cholesky
decomposition).

The spheric-radial decomposition allows us to rewrite the probability
function (\ref{probcons}) in the form%
\begin{equation}
\varphi (x)=\int\limits_{\mathbb{S}^{m-1}}e(x,v)d\mu _{\zeta }(v)\quad
\forall x\in X,  \label{fidef}
\end{equation}%
where $e:X\times \mathbb{S}^{m-1}\rightarrow \mathbb{R}$ 
refers to the \textit{radial probability function} defined by
\begin{equation}
e(x,v):=\mu _{\eta }\left( \left\{ r\geq 0\ |\ g(x,rLv)\leq 0\right\} \right).  \label{edef}
\end{equation}%

With any $x\in X$ satisfying $g(x,0)<0$, we will associate the finite and
infinite directions defined respectively as%
\begin{eqnarray*}
	F(x) &:&=\{v\in \mathbb{S}^{m-1}\ |\ \exists r\geq 0:g(x,rLv)=0\}, \\
	I(x) &:&=\{v\in \mathbb{S}^{m-1}\ |\ \forall r\geq 0:g(x,rLv)<0\}.
\end{eqnarray*}%
It is easily observed that $F(x)\cap I(x)=\emptyset $ and that $F(x)\cup
I(x)=\mathbb{S}^{m-1}$ by continuity of $g$. Moreover, the number $r\geq 0$
satisfying $g(x,rLv)=0$ in the case of $v\in F(x)$ is uniquely defined, due to the
convexity of $g$ in the second argument. This leads us to define the
following radius function for any $x$ with $g(x,0)<0$ and any $v\in \mathbb{S%
}^{m-1}$: 
\begin{equation}
\rho \left( x,v\right) :=\left\{ 
\begin{tabular}{ll}
$r$ such that $g(x,rLv)=0$ & if $v\in F(x)$ \\ 
$+\infty $ & if $v\in I(x)$.%
\end{tabular}%
\right.  \label{rhodef}
\end{equation}%
This definition allows us to rewrite the radial probability function $e$
from (\ref{edef}) in the form%
\begin{equation}
e(x,v)=\mu _{\eta }\left( \left[ 0,\rho \left( x,v\right) \right] \right)
=F_{\eta }\left( \rho \left( x,v\right) \right)  \label{enewdef}
\end{equation}%
whenever $g(x,0)<0$. Here, $F_{\eta }$ refers to the distribution function of the
Chi-distribution with $m$ degrees of freedom, so that $F^{\prime}_{\eta }(t)=\chi (t)$,
where $\chi $ is the corresponding density:
	\begin{equation}
	\chi \left( t\right) :=Kt^{m-1}e^{-t^{2}/2}\quad \forall t\geq 0.
	\label{chidens}
	\end{equation}%
The second equation in  (\ref{enewdef}) follows
from $F_{\eta }(0)=0$. We formally put $F_{\eta }(\infty ):=1$ which
translates the limiting property 
$F_{\eta }(t)\rightarrow _{t\rightarrow+\infty }1$
of cumulative distribution functions.

\subsection{Notation and tools from variational analysis}

\noindent Our notation will be standard. By $X$ and $X^{\ast }$ we will
denote a real reflexive and separable Banach space and its dual, 
with corresponding norms $\|\ \|$ and $\|\ \|_*$, and with corresponding balls $\mathbb{B}%
_{r}\left( x\right) $, $\mathbb{B}_{r}^{\ast }\left( x^{\ast }\right) $ of
radius $r$ around $x\in X$ and $x^{\ast }\in X^{\ast }$. 
We denote by $\langle x, x^*\rangle$, $x\in X, x^*\in X^*$ 
the corresponding duality product, and by  
$\rightharpoonup$ the weak convergence in both $X$ and $X^*$.
The negative polar
of some closed cone $C\subseteq X$ is the closed convex cone 
\begin{equation*}
	C^{\ast }:=\left\{ x^{\ast }\in X^{\ast }\left\langle x^{\ast
	},h\right\rangle \leq 0\quad \forall h\in C\right\} .
\end{equation*}
The notations  \textrm{cl\thinspace }$C$, \textrm{cl}$^{\ast }C$, $\mathrm{%
co\,}C$, and $\overline{\mathrm{co}}\,C$ will refer to the (strong or norm) closure,
the weak$^{\ast }$ closure, the
convex hull, and the closed convex hull  of $C\subseteq X$ (or $C\subseteq X^*$), respectively. 

The indicator and the support functions of a set $C\subseteq X$ (or $C\subseteq X^*$) are respectively defined as
$$
i_C(x):=0 \text { if } x\in C \text{ and } +\infty \text { otherwise, }
$$
$$
\sigma_C(x^*):=\sup_{x\in C}\langle x, x^*\rangle.
$$

\begin{definition}
	\noindent Let $C\subseteq X$ be a closed
	subset. Then the Fr\'{e}chet, the Mordukhovich, and the Clarke normal cones to $C$ at $%
\bar{x}\in C$ are respectively defined as
	\begin{align*}
		&N^{F}(\bar{x};C) :=\left\{ x^{\ast }\in X^{\ast }\ |\ \underset{x\rightarrow 
			\bar{x},x\in C}{\lim \sup }\frac{\left\langle x^{\ast },x-\bar{x}%
			\right\rangle }{\left\Vert x-\bar{x}\right\Vert }\leq 0\right\},\\
		&N^{M}(\bar{x};C) :=\left\{ x^{\ast }\in X^{\ast }\ |\ \exists x_{n}\rightarrow 
		\bar{x},\,x_{n}\in C,\,\exists x_{n}^{\ast }{\rightharpoonup }%
		x^{\ast }:x_{n}^{\ast }\in N^{F}(C;x_{n})\right\}, \\
		&N^{C}(\bar{x};C) :=\overline{\mathrm{co}}\, N^{M}(\bar{x};C).
	\end{align*}
\end{definition}

\noindent We note that the definition of $N^{C}$ is not the original but a
derived one. The normal cones induce subdifferentials of functions $%
f:X\rightarrow \mathbb{R}$ via their epigraphs 
\begin{equation*}
	\mathrm{epi\,}f:=\left\{ \left( x,t\right) \in X\times \mathbb{R}\ |\ f(x)\leq
	t\right\},
\end{equation*}%
which are closed whenever $f$ is lower semicontinuous (lsc, for short).

\begin{definition}
	\label{subdifferentials}Let 
	$f:X\rightarrow \mathbb{R}$ be a lsc  function. Then the Fr\'{e}chet, the Mordukhovich (limiting), and the Clarke subdifferentials of $%
	f $ at $\bar{x}\in X$, are respectively defined as%
	\begin{equation*}
		\partial ^{F/M/C}f\left( \bar{x}\right) :=\left\{ x^{\ast }\in X^{\ast
		}\ |\ \left( x^{\ast },-1\right) \in N^{F/M/C}(\left( \bar{x},f(\bar{x})\right) ;%
		\mathrm{epi\,}f)\right\}.
	\end{equation*}%
 The singular subdifferential of $f$ at $\bar{x}$ is defined as%
	\begin{equation*}
		\partial ^{\infty }f\left( \bar{x}\right) =\left\{ x^{\ast }\in X^{\ast
		}\ |\ \left( x^{\ast },0\right) \in N^{M}(\left( \bar{x},f(\bar{x})\right) ;%
		\mathrm{epi\,}f)\right\} .
	\end{equation*}
\end{definition}

\noindent We recall that the Fr\'{e}chet subdifferential has the explicit
representation%
\begin{equation}
\partial^{F}f\left( \bar{x}\right) =\left\{ x^{\ast }\in X^{\ast
}\ |\ \liminf_{x\rightarrow \bar{x}}\frac{f(x)-f\left( \bar{x}\right)
	-\left\langle x^{\ast },x-\bar{x}\right\rangle }{\left\Vert x-\bar{x}%
	\right\Vert }\geq 0\right\} .  \label{fresubalter}
\end{equation}

\noindent In the current setting of reflexive Banach spaces, the following
representation holds true for Clarke's subdifferential \cite[Theorem 3.57]%
{mordu}:%
\begin{equation}
\partial ^{C}f\left( \bar{x}\right) =\overline{\mathrm{co}}\, \left\{
\partial ^{M}f\left( \bar{x}\right) +\partial ^{\infty }f\left( \bar{x}%
\right) \right\}.  \label{clarkesub}
\end{equation}%
For locally Lipschitzian functions, the following classical definition of
Clarke's subdifferential applies:

\begin{equation}
\partial ^{C}f\left( \bar{x}\right) =\left\{ x^{\ast }\in X^{\ast
}\ |\ \left\langle x^{\ast },h\right\rangle \leq f^{\circ}\left( \bar{x};h\right)
,\, \forall h\in X\right\} ,  \label{genderiv}
\end{equation}%
where%
\begin{equation*}
	f^{\circ} \left( \bar{x};h\right) :=\underset{x\rightarrow \bar{x},t\downarrow 0}%
	{\lim \sup }\frac{f\left( x+th\right) -f(x)}{t}
\end{equation*}%
denotes Clarke's directional derivative of $f$ at $\bar{x}$ in the direction $h$.

In case that $f$ happens to be convex, all the  subdifferentials above coincide 
with the ordinary subdifferential  in the sense of convex analysis:
$$
\partial f\left(\bar{x}\right): =\left\{ 
x^{\ast }\in X^{\ast
}\ |\ f(x)\ge f( \bar{x})+\left\langle x^{\ast },x- \bar{x}\right\rangle
,\, \forall x\in X\right\}. 
$$

For a function $f(x,y)$ of two variables, we will refer to its partial
subdifferentials at a point $\left( \bar{x},\bar{y}\right) $ as the
corresponding subdifferentials of the partial functions:%
\begin{equation*}
	\partial _{x}^{F/M/C}f\left( \bar{x},\bar{y}\right) :=\partial
	^{F/M/C}f\left( \cdot ,\bar{y}\right) \left( \bar{x}\right) ;\quad \partial
	_{y}^{F/M/C}f\left( \bar{x},\bar{y}\right) :=\partial ^{F/M/C}f\left( \bar{x}%
	,\cdot \right) \left( \bar{y}\right) .
\end{equation*}%

\section{Continuity properties}

\noindent In this section, we investigate continuous properties of the radial probability and the radius 
functions, defined respectively in (\ref{edef}) and  (\ref{rhodef}), which are the basis for deriving in
Section \ref{subdiffprob} subdifferential formulae for probability
function  (\ref{probcons}). 

For all the following results, the
basic assumption (H) formulated in the Introduction is tacitly required to
hold; namely, function $g$ is locally Lipschitzian as a function of both arguments simultaneously, 
and convex as a function of the second argument.

\begin{lemma}
	\label{rholem}Define $U:=\{x\in X\ |\ g(x,0)<0\}$.
	
	\begin{enumerate}
		\item The radius function $\rho $ is continuous at $\left( x,v\right) $ for any $x\in U$ and any 
		$v\in F(x)$.
		
		\item For $x\in U$ and $v\in I(x)$ it holds that $\lim\limits_{k\rightarrow
			\infty }\rho \left( x_{k},v_{k}\right) =\infty $ for any sequence $\left(
		x_{k},v_{k}\right) \rightarrow (x,v)$ such that $v_{k}\in F(x_{k})$.
	\end{enumerate}
\end{lemma}
\begin{proof}
	Observe first, that $\rho $ is defined (possibly extended-valued) on $%
	U\times \mathbb{S}^{m-1}$. To verify 1., consider any sequence $\left(
	x_{k},v_{k}\right) \rightarrow _{k}\left( x,v\right) $ with $v_{k}\in 
	\mathbb{S}^{m-1}$. We show first that the sequence $\rho \left(
	x_{k},v_{k}\right) $ is bounded. Indeed, otherwise there would exist a
	subsequence with $\rho \left( x_{k_{l}},v_{k_{l}}\right) \rightarrow
	_{l}\infty $. Clearly $g(x_{k_{l}},0)<0$ for $l$ large enough, because of $%
	g(x,0)<0$. Fix an arbitrary $r\geq 0$. Then $\rho \left(
	x_{k_{l}},v_{k_{l}}\right) >r$. We claim that $g(x_{k_{l}},rLv_{k_{l}})<0$
	for these $l$'s. This is obvious in case that $v_{k_{l}}\in I(x_{k_{l}})$. If $%
	v_{k_{l}}\in F(x_{k_{l}})$, then the relations%
	\begin{equation*}
		g(x_{k_{l}},0)<0,\quad g(x_{k_{l}},\rho \left( x_{k_{l}},v_{k_{l}}\right)
		Lv_{k_{l}})=0,\quad \rho \left( x_{k_{l}},v_{k_{l}}\right) >r,
			\end{equation*}%
	and 
	$$
	g(x_{k_{l}},rLv_{k_{l}})\geq 0,
	$$
	would contradict the convexity of $g$ in the second argument. Hence, 
	for $l$ sufficiently large, 
	$$
	g(x_{k_{l}},rLv_{k_{l}})<0,
	$$
	and passing to the
	limit yields that $g(x,rLv)\leq 0$, which holds true for all $r\geq 0$
	because the latter was chosen arbitrary. But then, $g(x,rLv)<0$ for all $%
	r\geq 0$, because otherwise once more a contradiction with convexity of $g$
	in the second argument would arise from $g(x,0)<0$. This, however, amounts
	to $v\in I(x)$ contradicting our assumption $v\in F(x)$. Summarizing, we
	have shown that $\rho \left( x_{k},v_{k}\right) $ is bounded and, in
	particular, $v_{k}\in F(x_{k})$ for all $k$. Let $\rho \left(
	x_{k_{l}},v_{k_{l}}\right) \rightarrow _{l}r_0$ be an arbitrary convergent
	subsequence. Then, we may pass to the limit in the relation $g\left(
	x_{k_{l}},\rho \left( x_{k_{l}},v_{k_{l}}\right) Lv_{k_{l}}\right) =0$ in
	order to derive that $g\left( x,r_0Lv\right) =0$, which in turn implies that $%
	r_0=\rho \left( x,v\right) $. Hence, all convergent subsequences of $\rho
	\left( x_{k},v_{k}\right) $ have the same limit $\rho \left( x,v\right) $.
	This implies that $\rho \left( x_{k},v_{k}\right) \rightarrow _{k}\rho
	\left( x,v\right) $ and altogether that $\rho $ is continuous at $\left(
	x,v\right) $.
	
	As for 2., observe that if $\rho \left( x_{k},v_{k}\right) $ would not tend
	to infinity, then there would exist a converging subsequence $\rho \left(
	x_{k_{l}},v_{k_{l}}\right) \rightarrow _{l}r_1$ for some $r_1\geq 0$. Since $%
	\rho \left( x_{k_{l}},v_{k_{l}}\right) <\infty $ and $g(x_{k_{l}},0)<0$ for $%
	l$ large enough, we infer that $v_{k_{l}}\in F(x_{k_{l}})$ and, hence, $%
	g(x_{k_{l}},\rho \left( x_{k_{l}},v_{k_{l}}\right) Lv_{k_{l}})=0$ for all
	these $l$'s. Now, passing to the limit yields that $g(x,r_1Lv)=0$, whence $v\in
	F(x)$, a contradiction.
\end{proof}

\begin{lemma}
	\label{opendom}If $g\left( x,0\right) <0$ and $v\in F(x)$, then there exist
	neighborhoods $U$ and $V$ of $x$ and $v$, respectively, such that $%
	v^{\prime }\in F(x^{\prime })$ for all $x^{\prime }\in U$ and $v^{\prime
	}\in V\cap \mathbb{S}^{m-1}$.
\end{lemma}

\begin{proof}
	If the statement wasn't true, then there existed a sequence $\left(
	x_{k},v_{k}\right) \rightarrow (x,v)$ with $g\left( x_{k},0\right) <0$, $%
	v_{k}\in \mathbb{S}^{m-1}$ and $v_{k}\in I\left( x_{k}\right) $. Hence, $%
	\rho \left( x_{k},v_{k}\right) =\infty $ and so $\rho \left( x,v\right)
	=\infty $ by 1. in Lemma \ref{rholem}. This yields the contradiction $v\in
	I(x)$.
\end{proof}

\begin{lemma}
	\label{lowest}Let $x\in X$ and $r\geq 0$ be such that 
         $g(x,0)<0$ and $g(x,rLv)=0$. Then%
	\begin{equation*}
		\left\langle z^{\ast },Lv\right\rangle \geq -\frac{g(x,0)}{r}>0\quad \forall
		z^{\ast }\in \partial _z g\left( x,rLv\right).
	\end{equation*}%
	\end{lemma}
	
\begin{proof}
	By convexity of $g$ in the second variable and by definition of the convex
	subdifferential, one has that 
	\begin{eqnarray*}
		-\frac{r}{2}\left\langle z^{\ast },Lv\right\rangle &=&\left\langle z^{\ast },%
		\frac{r}{2}Lv-rLv\right\rangle \leq g\left( x,\frac{r}{2}Lv\right) -g\left(
		x,rLv\right) \\
		&=&g\left( x,\frac{r}{2}Lv\right) \leq \frac{1}{2}g\left( x,0\right) +\frac{1%
		}{2}g\left( x,rLv\right) =\frac{1}{2}g\left( x,0\right) .
	\end{eqnarray*}%
	Since our assumptions imply that $r>0$, the assertion follows.
\end{proof}

We get in the following proposition the desired continuity of the radial probability
function $e$ defined in (\ref{edef}).
\begin{proposition}
	\label{econt}The radial probability function is
	continuous at any $\left( x,v\right) \in X\times \mathbb{S}^{m-1}$ with $%
	g(x,0)<0$.
\end{proposition}

\begin{proof}
	Fix a point $\left( x,v\right) \in X\times \mathbb{S}^{m-1}$ with $%
	g(x,0)<0$. Consider any sequence $\left( x_{k},v_{k}\right) \rightarrow (x,v)$ with $%
	v_{k}\in \mathbb{S}^{m-1}$ and assume first that $v\in F(x)$. Then, $\rho
	\left( x_{k},v_{k}\right) \rightarrow _{k}\rho \left( x,v\right) $ by 1. in
	Lemma \ref{rholem}, and $v_{k}\in F(x_{k})$ for $k$
	large, by Lemma \ref{opendom}. Hence, by (\ref{enewdef}) it follows that 
	\begin{equation*}
		e\left( x_{k},v_{k}\right) =F_{\eta }\left( \rho (x_{k},v_{k})\right)
		\rightarrow _{k}F_{\eta }(\rho \left( x,v\right) )=e\left( x,v\right) ,
	\end{equation*}%
	where the convergence follows from the continuity of the Chi-distribution
	function $F_{\eta }$.
	
	If in contrast $v\in I(x)$, then, by (\ref{edef}), $e\left( x,v\right) =\mu
	_{\eta }\left( \mathbb{R}_{+}\right) =1$. We'll be done if we can show that $%
	e\left( x_{k},v_{k}\right) \rightarrow _{k}1$. If this did not hold true,
	then there would exist a subsequence and some $\varepsilon >0$ such that 
	\begin{equation}
	\left\vert e\left( x_{k_{l}},v_{k_{l}}\right) -1\right\vert >\varepsilon
	\quad \forall \ l\text{.}  \label{subseq}
	\end{equation}%
	Since $v_{k_{l}}\in I\left( x_{k_{l}}\right) $ would imply as above that $%
	e\left( x_{k_{l}},v_{k_{l}}\right) =\mu _{\eta }\left( \mathbb{R}_{+}\right)
	=1$, a contradiction, we conclude that $v_{k_{l}}\in F\left(
	x_{k_{l}}\right) $ for all $l$. Now, 2. in Lemma \ref{rholem} guarantees
	that $\rho \left( x_{k_{l}},v_{k_{l}}\right) \rightarrow _{l}\infty $. Then,
	by (\ref{enewdef}), we arrive at the convergence 
	\begin{equation*}
		e\left( x_{k_{l}},v_{k_{l}}\right) =F_{\eta }\left( \rho
		(x_{k_{l}},v_{k_{l}})\right) \rightarrow _{l}1,
	\end{equation*}%
	where we exploited the property $\lim\limits_{t\rightarrow \infty }F_{\eta
	}\left( t\right) =1$, following from $F_{\eta }$ being a cumulative
	distribution function. This is a contradiction with (%
	\ref{subseq}), and the desired conclusion follows.
\end{proof}
Consequently, we obtain the continuity of the probability function $\varphi $,
defined in (\ref{probcons}).
\begin{theorem}
	The probability function is continuous at any point $x\in X$ with $g(x,0)<0$.
\end{theorem}

\begin{proof}
	For any sequence $x_{n}\rightarrow x$ one has by Proposition \ref{econt} that%
	\begin{equation*}
		e\left( x_{n},v\right) \rightarrow _{n}e\left( x,v\right) \leq 1\quad
		\forall v\in \mathbb{S}^{m-1},
	\end{equation*}%
	where the inequality follows from $e$ being a probability. Since the constant function 
	$1$ is integrable on $\mathbb{S}^{m-1}$, the assertion follows from Lebesgue's
	dominated convergence theorem.
\end{proof}

\section{Subdifferential of the radial probability function}

\noindent In this section, we provide characterizations of the Fr\'{e}chet
subdifferential of the radial probability function $e\left( \cdot
,v\right) $, defined in (\ref{edef}), for arbitrarily fixed directions $v\in 
\mathbb{S}^{m-1}$.  As before, we also consider in this section
our standard assumption (H).


\

We need first to estimate the set $ \partial _{x}^{F}\rho (x,v)$:
\begin{proposition}
	\label{frechetrho} Let $x\in X$ with $g(x,0)<0\,$and $\,v\in F(x)$ be
	arbitrary. Then, for every $y^{\ast }\in \partial _{x}^{F}\rho (x,v)$ and
	every $w\in X$, there exist $x^{\ast }\in \partial _{x}^{C}g(x,\rho (x,v)Lv)$
	and $z^{\ast }\in \partial_z g\left( x,\rho (x,v)Lv\right) $ such that 
	$\left\langle z^{\ast},Lv\right\rangle>0$ and 
	\begin{equation*}
		\left\langle y^{\ast },w\right\rangle \leq \frac{-1}{\left\langle z^{\ast
			},Lv\right\rangle }\left\langle x^{\ast },w\right\rangle .
	\end{equation*}
\end{proposition}

\begin{proof}
	Fix  $y^{\ast }\in \partial _{x}^{F}\rho (x,v)$ and $w\in X$; hence, $\rho
	(x,v)<\infty $ (because by assumption $v\in F(x)$). Let $M>0$ be a Lipschitz
	constant of $g$ at $\left( x,\rho (x,v)Lv\right)$. Then, there exists a
	neighborhood $U$ of $x$ such that the function $g(\cdot ,\rho (x,v)Lv)\,$is
	locally Lipschitzian with Lipschitz constant $M$ at each $x^{\prime }\in U$, and 
	such that the
	functions $g(x^{\prime },\cdot )$, $x^{\prime }\in U$, are locally Lipschitzian with 
	the same Lipschitz constant $M$ at $\rho (x,v)Lv$. As a consequence
	of \cite[Proposition 2.1.2]{clarke}, for all  $x^{\prime }\in U$ one has that 
	\begin{equation}
	\left\Vert x^{\ast }\right\Vert,\left\Vert z^{\ast }\right\Vert \leq M\quad
	\forall x^{\ast }\in \partial _{x}^{C}g(x^{\prime },\rho (x,v)Lv),\,\,\forall
	z^{\ast }\in \partial_z g(x^{\prime },\rho (x,v)Lv).  \label{boundedset}
	\end{equation}%
	Consider an arbitrary sequence $t_{n}\downarrow 0$ so that, by Lemma \ref%
	{opendom}, we may assume $v\in F(x+t_{n}w)$ for all $n$. By convexity and 
	continuity of the function $g$ with respect to the second variable, 
	the set $\partial g\left(
	x+t_{n}w,\cdot \right) (\rho (x,v)Lv)$ is nonempty for all $n$, and so we
	may select a sequence 
\begin{equation}
	z_{n}^{\ast }\in \partial_z g\left( x+t_{n}w,\cdot \right) (\rho (x,v)Lv);
	\label{convsub}
	\end{equation}%
	hence, taking into account, from the definition of function $\rho$,  that 	
	$g(x+t_{n}w,\rho (x+t_{n}w,v)Lv)=0$ and $g(x,\rho (x,v)Lv)=0$,
		\begin{align}
		\left( \rho (x+t_{n}w,v)-\rho (x,v)\right) \left\langle z_{n}^{\ast
		},Lv\right\rangle 
		&=\left\langle z_{n}^{\ast },\rho (x+t_{n}w,v)Lv-\rho (x,v)Lv\right\rangle\nonumber\\
		&\leq g(x+t_{n}w,\rho (x+t_{n}w,v)Lv)\nonumber\\
		&\qquad\qquad\qquad -g(x+t_{n}w,\rho (x,v)Lv)\nonumber\\
		&=-g(x+t_{n}w,\rho (x,v)Lv)\nonumber\\
		&= g(x,\rho (x,v)Lv)-g(x+t_{n}w,\rho (x,v)Lv)\label{oo}.
	\end{align}
	Next, Lebourg's mean value Theorem for Clarke's subdifferential \cite[Theorem
	2.3.7]{clarke} yields   some $\tau _{n}\in \left[ 0,1\right] $
	and
	\begin{equation}
	x_{n}^{\ast }\in \partial _{x}^{C}g(x+\tau _{n}t_{n}w,\rho (x,v)Lv)
	\label{meanvalue}
	\end{equation}%
	such that 
	\begin{equation}
	g(x,\rho (x,v)Lv)-g(x+t_{n}w,\rho (x,v)Lv) \leq -t_{n}\left\langle x_{n}^{\ast },w\right\rangle,
	\label{uppest3}
	\end{equation}%
and, consequently, from (\ref{oo}),
	\begin{equation}
	\left( \rho (x+t_{n}w,v)-\rho (x,v)\right) \left\langle z_{n}^{\ast
	},Lv\right\rangle \leq -t_{n}\left\langle x_{n}^{\ast },w\right\rangle .
	\label{uppest}
	\end{equation}%
	Since $X$ is reflexive and $\left\Vert z_{n}^{\ast }\right\Vert,\ \left\Vert x_{n}^{\ast }\right\Vert \leq M$, 
	by (\ref{boundedset}), there exists a
	subsequence $\left( x_{n_{k}}^{\ast },z_{n_{k}}^{\ast }\right) $ and some $%
	\left( x^{\ast },z^{\ast }\right) \in X\times \mathbb{R}^{m}$ such that $%
	x_{n_{k}}^{\ast }\rightharpoonup x^{\ast }$ and $z_{n_{k}}^{\ast
	}\rightarrow z^{\ast }$. The weak$^{\ast }$-closedness of the graph of
	Clarke's subdifferential \cite[Proposition 2.1.5]{clarke} along with (\ref%
	{meanvalue}) and (\ref{convsub}) implies that%
	\begin{equation}
	x^{\ast }\in \partial _{x}^{C}g(x,\rho (x,v)Lv),\ z^{\ast }\in \partial _{z}g\left( x,\rho (x,v)Lv\right).  \label{partclarkex}
	\end{equation}%
	Now, Lemma \ref{lowest} implies that%
	\begin{equation*}
		\left\langle z^{\ast },Lv\right\rangle \geq \frac{-g(x,0)}{\rho (x,v)}>0,
	\end{equation*}%
	and, so, by passing to the (inferior) limit in (\ref{uppest}), we arrive at%
	\begin{equation}
	 \left\langle z^{\ast },Lv\right\rangle\liminf_{n\rightarrow \infty }t_{n}^{-1}\left( \rho (x+t_{n}w,v)-\rho
	(x,v)\right) \leq -\left\langle
	x^{\ast },w\right\rangle .  \label{uppest2}
	\end{equation}%
	Therefore, since $y^{\ast }\in \partial _{x}^{F}\rho(x,v)$, 
	\begin{equation*}
	\left\langle y^{\ast },w\right\rangle \le \liminf_{n\rightarrow \infty
	}t_{n}^{-1}\left( \rho (x+t_{n}w,v)-\rho (x,v)\right) \le \frac{-1}{\left\langle z^{\ast },Lv\right\rangle}\left\langle x^{\ast },w\right\rangle,
	\end{equation*}%
	as we wanted to prove.
\end{proof}

Next, we give the desired estimate of the set $ \partial _{x}^{F}e (x,v)$. Recall that $\chi $ is the density of the one-dimensional Chi-distribution with $m$ degrees of freedom (see (\ref{chidens})).
\begin{theorem}
	\label{chainrule}Let $x\in X$ with $g(x,0)<0\,$and $\,v\in F(x)$ be
	arbitrary. Then, for every $y^{\ast }\in \partial _{x}^{F}e(x,v)$ and every $%
	w\in X$, there exist $x^{\ast }\in \partial _{x}^{C}g(x,\rho (x,v)Lv)$ and $%
	z^{\ast }\in \partial_z g\left( x,\rho (x,v)Lv\right) $ such that%
	\begin{equation*}
		\left\langle y^{\ast },w\right\rangle \leq \frac{-\chi \left( \rho
			(x,v)\right) }{\left\langle z^{\ast },Lv\right\rangle }\left\langle x^{\ast
		},w\right\rangle .
	\end{equation*}%
	Consequently,  if $M_{x,v}$ denotes a Lipschitz constant of  $g(\cdot ,\rho
	(x,v)Lv)$ at $x$, then 
	\begin{equation*}
		\left\Vert y^{\ast }\right\Vert \leq \frac{\rho (x,v)\cdot \chi \left( \rho
			(x,v)\right) }{\left\vert g(x,0)\right\vert }M_{x,v}\quad \forall y^{\ast
		}\in \partial _{x}^{F}e(x,v).
	\end{equation*}%
	\end{theorem}

\begin{proof}
	By (\ref{enewdef}), for all $y$ close to $x$ we may write $e\left( y,v\right) =F_{\eta }\left( \rho
	(y,v)\right) $, with $\rho (y,v)<\infty $,
	as a consequence of Lemma \ref{opendom}. Since $F_{\eta }$ is continuously
	differentiable and nondecreasing (as a distribution function), $%
	F_{\eta }^{\prime }\left( t\right) \geq 0$ for all $t\in \mathbb{R}$ and, from
	the calculus of Fr\'{e}chet subdifferentials (e.g., \cite[Corollary 1.14.1 and
	Proposition 1.11]{kruger}), we obtain that 
	\begin{eqnarray*}
		\partial _{x}^{F}e(x,v) &=&\partial
		^{F}\left( F_{\eta }^{\prime }\left( \rho (x,v)\right)  \rho (\cdot
		,v)\right) (x) \\
		&=&F_{\eta }^{\prime }(\rho (x,v)) \partial ^{F}\rho (\cdot
		,v) (x)=\chi \left( \rho (x,v)\right) \partial _{x}^{F}\rho
		(x,v).
	\end{eqnarray*}%
	Combination with Proposition \ref{frechetrho} yields the first assertion.

To prove the second assertion, from the first part of the proposition we choose
	elements $x^{\ast }\in \partial _{x}^{C}g(x,\rho (x,v)Lv)$ and $z^{\ast }\in
	\partial_z g\left( x,\rho (x,v)Lv\right) $ such that 
	\begin{equation*}
		\left\langle y^{\ast },w\right\rangle \leq \left\vert \frac{-\chi \left(
			\rho (x,v)\right) }{\left\langle z^{\ast },Lv\right\rangle }\right\vert
		\left\Vert x^{\ast }\right\Vert \left\Vert w\right\Vert,
	\end{equation*}%
	and so, since 
	$
		\left\langle z^{\ast },Lv\right\rangle \geq \frac{-g(x,0)}{\rho (x,v)}>0
	$
	by Lemma \ref{lowest}, 
		\begin{equation*}
		\left\langle y^{\ast },w\right\rangle \leq \frac{\rho (x,v)\cdot \chi \left(
			\rho (x,v)\right) }{\left\vert g(x,0)\right\vert}M_{x,v}\left\Vert
		w\right\Vert ,
	\end{equation*}%
	yielding the desired conclusion.
\end{proof}

We shall also need the following result.
\begin{corollary}\label{loclipe}
{\rm (i)} For every $x_{0}\in X$ with $g(x_{0},0)<0\,$\ and every $%
	\,v_{0}\in F(x_{0})$ there exist neighborhoods $\tilde{U}$ of $x_{0}$ and $%
	\tilde{V}$ of $v_{0}$ as well as some $\alpha >0$ such that 
	\begin{equation}
	\partial _{x}^{F}e(x,v)\subseteq \mathbb{B}_{\alpha }^{\ast }\left( 0\right)
	\quad \forall \left( x,v\right) \in \tilde{U}\times \left( \tilde{V}\cap 
	\mathbb{S}^{m-1}\right) .  \label{rel1}
	\end{equation}%
	\par\noindent {\rm (ii)} For all $x\in X$ with $g(x,0)<0\,$and for all $\,v\in I(x)$
	one has that $\partial _{x}^{F}e(x,v)\subseteq \left\{ 0\right\} $.
\end{corollary}

\begin{proof}
	(i) Let $M>0$ and define open neighborhoods $\tilde{U}$ of $x_{0}$ and $\tilde{V}$ of $v_{0}$ such that
	$M$ is a Lipschitz constant of $g$ on $\tilde{U}\times \tilde{V}$ and, for all $(x,v)\in \tilde{U}\times \left( 
		\tilde{V}\cap \mathbb{S}^{m-1}\right) $ (recall Lemma \ref{opendom}),
		$$
		g(x,0)<0,\  \rho (x,v)<\infty.
		$$
	 	
	\noindent Hence, by Theorem \ref{chainrule},%
	\begin{equation*}
		\partial _{x}^{F}e(x,v)\subseteq \mathbb{B}_{\alpha (x,v)}^{\ast }\left(
		0\right) ,
	\end{equation*}%
	where 
	$$
	\alpha (x,v):=\frac{\rho (x,v)\cdot \chi \left( \rho
			(x,v)\right) }{\left\vert g(x,0)\right\vert }M_{x,v}.
			$$
	Taking into account the continuity of $\rho$ (see Lemma \ref{rholem}), 
	we may suppose for all $\left( x,v\right) \in \tilde{U}\times \left( \tilde{V}\cap \mathbb{S%
	}^{m-1}\right)$ that $M$ is  a Lipschitz constant for 
	$g(\cdot ,\rho (x,v)Lv)$ at the point $x\ (\in \tilde{U})$. 
	Thus, we can replace $M_{x,v}$  by $M$
	in the definition of $\alpha $ above. Moreover, since $g$ is continuous (also by 
	Lemma \ref{rholem}), as well as  the Chi-density $%
	\chi $, we deduce that $\alpha $ is continuous on $\tilde{U}%
	\times \left( \tilde{V}\cap \mathbb{S}^{m-1}\right) $. Then, after
	shrinking $\tilde{U}\times \tilde{V}$ if necessary, we may assume that  for some $\alpha >0$
	$$
	\alpha (x,v)\leq \alpha  \qquad \forall \left( x,v\right) \in \tilde{U}\times
	\left( \tilde{V}\cap \mathbb{S}^{m-1}\right).
          $$ 
          This proves (\ref{rel1}).

(ii)	As already observed in the proof of Proposition \ref{econt}, $v\in I(x)$
	implies that $e(x,v)=1$. Consequently, the function $e(\cdot ,v)$ (as the
	value of a probability) reaches a global maximum at $x$. Let $x^{\ast }\in
	\partial _{x}^{F}e(x,v)$ and $u\in X\backslash \{0\}$ be arbitrary. Then,%
	\begin{align*}
		-\left\langle x^{\ast },\frac{u}{\left\Vert u\right\Vert }\right\rangle
		&=\liminf_{n\rightarrow \infty }-\frac{\left\langle x^{\ast
			},n^{-1}u\right\rangle }{\left\Vert n^{-1}u\right\Vert }\\
		&\geq  \liminf_{n\rightarrow \infty }\frac{e(x+n^{-1}u,v)-e(x,v)-\left\langle x^{\ast
			},n^{-1}u\right\rangle }{\left\Vert n^{-1}u\right\Vert } \\
		&\geq \liminf_{h\rightarrow 0}\frac{e(x+h,v)-e(x,v)-\left\langle x^{\ast
			},h\right\rangle }{\left\Vert h\right\Vert }\geq 0.
	\end{align*}%
	Hence $\left\langle x^{\ast },u\right\rangle \leq 0$ for all $u\in X$, and
	so $x^{\ast }=0$ as desired.
\end{proof}

\begin{definition}
	\label{nicedirect}For $x\in X$ and $l>0$, we call%
	\begin{equation*}
		C_{l}(x):=\{h\in X\ |\ g^{\circ}(\cdot ,z)(y;h)\leq l \left\Vert z\right\Vert
		^{-m}e^{\frac{\left\Vert z\right\Vert^2}{2\left\Vert L\right\Vert^2} }\left\Vert h\right\Vert \ \forall
		y\in \mathbb{B}_{1/l}\left( x\right), \, \left\Vert z\right\Vert
		\geq l\}
	\end{equation*}%
	the $l$-\it{cone of nice directions} at $x\in X$. We denote the polar cone to $C_{l}(x)$ as $C^{*}_{l}(x)$.
	\end{definition}

\noindent Note that, by positive homogeneity of Clarke's directional
derivative, $\{C_{l}\}_{l\in \mathbb{N}}$
defines a nondecreasing sequence of closed cones.

We give in the following theorem another estimate for $\partial _{x}^{F}e(x,v)$, which will be useful in the sequel.
\begin{theorem}
	\label{domin}Fix $x_{0}\in X$ such that $g(x_{0},0)<0$. Then, for every $l>0$,
	there exists some neighborhood $U$ of $x_{0}$ and some  $R>0$ such that%
	\begin{equation*}
		\partial _{x}^{F}e(x,v)\subseteq \mathbb{B}_{R}^{\ast }\left( 0\right)
		-C_{l}^{\ast }(x_{0})\quad \forall x\in U,\, v\in \mathbb{S}^{m-1}.
	\end{equation*}%
\end{theorem}

\begin{proof}
	Let $l>0$ be arbitrarily fixed. It will be sufficient to show that for every 
	$v_{0}\in \mathbb{S}^{m-1}$ there are neighborhoods $\bar{U}$ of $x_{0}$
	and $\bar{V}$ of $v_{0}$ and some  $R>0$ such that 
	\begin{equation}
	\partial _{x}^{F}e(x,v)\subseteq \mathbb{B}_{R}^{\ast }\left( 0\right)
	-C_{l}^{\ast }(x_{0})\quad \forall \left( x,v\right) \in \bar{U}\times (\bar{%
		V}\cap \mathbb{S}^{m-1}).  \label{claimrel}
	\end{equation}%
	If this holds true, then the global inclusion in the statement of this
	proposition will follow from the local ones above by a standard compactness
	argument with respect to $\mathbb{S}^{m-1}$.
	
	In order to prove (\ref{claimrel}), fix an arbitrary $v_{0}\in \mathbb{S}%
	^{m-1}$. Assume first that $v_{0}\in I(x_{0})$. Then, define open
	neighborhoods $U^{\ast }$ of $x_{0}$ and $V^{\ast }$ of $v_{0}$ such that
	$U^{\ast }\subseteq \mathbb{B}_{1/l}\left( x_{0}\right)$ (with $l>0$ as fixed above)
	and,  for all  $x\in U^{\ast }$ and $v\in V^{\ast }\cap F(x)$,
	$$
		g(x,0)\leq \frac{1}{2}g(x_{0},0)<0,\ 		
		\rho (x,v)\Vert Lv\Vert\geq l.
	$$	
	\noindent Note, that the last inequality is possible by virtue of 2. in Lemma %
	\ref{rholem} and by $L$ being nonsingular and $\mathbb{S}^{m-1}$ being
	compact (therefore $\Vert Lv\Vert\geq \delta $ for all $v\in \mathbb{S}^{m-1}$ and some 
	$\delta >0$). From Corollary \ref{loclipe}(ii) we derive that%
	\begin{equation}
	\partial _{x}^{F}e(x,v)\subseteq \left\{ 0\right\} \quad \forall x\in
	U^{\ast },\,\, v\in I(x).  \label{trivinclu}
	\end{equation}%
	Now, consider an arbitrary $\left( x,v\right) \in U^{\ast }\times
	V^{\ast }$ such that $v\in F(x)$. Let also $y^{\ast }\in \partial
	_{x}^{F}e(x,v)$ and $h\in -C_{l}(x_{0})$ be arbitrarily given. Then, by
	Theorem \ref{chainrule}, there exist $x^{\ast }\in \partial
	_{x}^{C}g(x,\rho (x,v)Lv)$ and $z^{\ast }\in \partial _{z}g\left( x,\rho
	(x,v)Lv\right) $ such that%
	\begin{equation}
	\left\langle y^{\ast },h\right\rangle \leq \frac{\chi \left( \rho
		(x,v)\right) }{\left\langle z^{\ast },Lv\right\rangle }\left\langle x^{\ast
	},-h\right\rangle \leq \frac{\chi \left( \rho (x,v)\right) }{\left\langle
		z^{\ast },Lv\right\rangle }g^{\circ}(\cdot ,\rho (x,v)Lv)(x;-h),
	\label{intermed2}
	\end{equation}%
	where the last inequality relies on (\ref{genderiv}) and on the fact that
	both the density function $\chi $ and $\left\langle z^{\ast },Lv\right\rangle $ are
	positive (see Lemma \ref{lowest}). Since $-h\in C_{l}(x_{0})$, our
	conditions on the neighborhoods $U^{\ast }$ and $V^{\ast }$ stated above
	guarantee that%
	\begin{eqnarray*}
		g^{\circ}(\cdot ,\rho (x,v)Lv)(x;-h)&\leq& l \left\Vert \rho
		(x,v)Lv\right\Vert ^{-m}e^{\frac{\left\Vert \rho (x,v)Lv\right\Vert^2}{2\Vert L\Vert^2}}\left\Vert h\right\Vert\\
		&\leq& l \left\Vert \rho
		(x,v)Lv\right\Vert ^{-m}e^{\frac{\rho (x,v)^2}{2}}\left\Vert h\right\Vert .
	\end{eqnarray*}
	This allows us to continue (\ref{intermed2}) as%
	\begin{eqnarray*}
		\left\langle y^{\ast },h\right\rangle 
		&\leq &\frac{\chi \left( \rho (x,v)\right)  \rho (x,v)l}{\left\vert
			g(x,0)\right\vert } \left\Vert \rho (x,v)Lv\right\Vert
		^{-m}e^{ \frac{\rho (x,v)^2}{2} }\left\Vert h\right\Vert \\
		&=&\frac{lK}{\left\vert g(x,0)\right\vert }\left\Vert Lv\right\Vert
		^{-m}\left\Vert h\right\Vert ,
	\end{eqnarray*}%
	where we used Lemma \ref{lowest} and the definition of the Chi-density with $%
	m$ degrees of freedom (see (\ref{chidens})).
	Owing to $g(x,0)\leq \frac{1}{2}g(x_{0},0)<0$, we may continue as 
	\begin{equation}\label{ConstantR}
		\left\langle y^{\ast },h\right\rangle \leq \frac{2lKK^{\ast }}{\left\vert
			g(x_0,0)\right\vert }\left\Vert h\right\Vert ,
	\end{equation}%
	where (recall that $L$ is nonsingular)%
	\begin{equation*}
		K^{\ast }:=\max_{v\in \mathbb{S}^{m-1}}\left\Vert Lv\right\Vert
		^{-m}\in \mathbb{R}_+.
	\end{equation*}%
	Consequently, we have shown that for some $\tilde{K}>0$, which is independent of $x$ and $v$,%
	\begin{equation*}
		\left\langle y^{\ast },h\right\rangle \leq \tilde{K}\left\Vert h\right\Vert
		\quad \forall y^{\ast }\in \partial _{x}^{F}e(x,v),\, h\in
		-C_{l}(x_{0}).
	\end{equation*}%
	Using indicator and support functions, respectively, this relation is
	rewritten as, for all $ h\in X$,
	\begin{eqnarray*}
		\left\langle y^{\ast },h\right\rangle &\leq &\tilde{K}\left\Vert
		h\right\Vert +i_{-\overline{\mathrm{co}}\,C_{l}(x_{0})}(h)\\
		&=&\sigma _{%
			\mathbb{B}_{\tilde{K}}^{\ast }\left( 0\right) }(h)+\sigma _{-C_{l}^{\ast
			}(x_{0})}(h) \\
		&=&\sigma _{\left(\mathbb{B}_{\tilde{K}}^{\ast }\left( 0\right) -C_{l}^{\ast
			}(x_{0})\right)}(h).	\end{eqnarray*}%
	Consequently, we get
	\begin{equation*}
		\sigma _{\partial _{x}^{F}e(x,v)}(h)\leq \sigma _{\left(\mathbb{B}_{\tilde{K}%
			}^{\ast }\left( 0\right) -C_{l}^{\ast }(x_{0})\right)}(h)\quad \,\,\forall h\in X,
	\end{equation*}%
	which entails the inclusion%
	\begin{equation*}
		\partial _{x}^{F}e(x,v)\subseteq \mathbb{B}_{\tilde{K}}^{\ast }\left(
		0\right) -C_{l}^{\ast }(x_{0}).
	\end{equation*}%
         Since $%
	\left( x,v\right) \in U^{\ast }\times V^{\ast }$ with $v\in F(x)$ were
	chosen arbitrarily, we may combine this with (\ref{trivinclu}) to
	derive that 
	\begin{equation*}
		\partial _{x}^{F}e(x,v)\subseteq \mathbb{B}_{\tilde{K}}^{\ast }\left(
		0\right) -C_{l}^{\ast }(x_{0})\quad \forall \left( x,v\right) \in U^{\ast
		}\times \left( V^{\ast }\cap \mathbb{S}^{m-1}\right).
	\end{equation*}%
	
	Now, we suppose that $v_{0}\in F(x_{0})$. Then Corollary \ref{loclipe}(i)
	guarantees the existence of neighborhoods $\tilde{U}$ of $x_{0}$ and $%
	\tilde{V}$ of $v_{0}$ as well as some $\alpha >0$ such that relation (\ref%
	{rel1}) holds true. Consequently, we end up with the claimed relation (\ref%
	{claimrel}) upon putting%
	\begin{equation*}
		\bar{U}:=\tilde{U}\cap U^{\ast },\,\,\bar{V}:=\tilde{V}\cap V^{\ast
		},\,\,R:=\max \{\alpha ,\tilde{K}\}.
	\end{equation*}%
	\hfill
\end{proof}

\begin{corollary}
	\label{uniformlipschitz}Fix $x_{0}\in X$ such that $g(x_{0},0)<0$, and assume
	one of the following alternative conditions:%
	\begin{equation}
	\left\{ z\in \mathbb{R}^{m}\ |\ g\left( x_{0},z\right) \leq 0\right\}  
	\mathrm{\mbox{is a bounded set}},  \label{bounded}
	\end{equation}%
	or%
	\begin{equation}
	\exists \ l>0 \text{ such that } C_{l}(x_{0})=X.  \label{allnice}
	\end{equation}%
	Then the partial radial probability functions $e(\cdot ,v)$, $v\in \mathbb{S}^{m-1}$,  are uniformly locally Lipschitzian
	around $x_{0}$ with some common Lipschitz constant
	independent of $v$.
\end{corollary}

\begin{proof}
	In the case of (\ref{bounded}), one has that $I(x_{0})=\emptyset $, whence $%
	F(x_{0})=\mathbb{S}^{m-1}$. Then, by Corollary \ref{loclipe}(i), for every $%
	\,v_{0}\in \mathbb{S}^{m-1}$ there exist neighborhoods $\tilde{U}_{v_{0}}$
	of $x_{0}$ and $\tilde{V}_{v_{0}}$ of $v_{0}$ as well as some $\alpha
	_{v_{0}}>0$ such that%
	\begin{equation*}
		\partial _{x}^{F}e(x,v)\subseteq \mathbb{B}_{\alpha _{v_{0}}}^{\ast }\left(
		0\right) \quad \forall \left( x,v\right) \in \tilde{U}_{v_{0}}\times \left( 
		\tilde{V}_{v_{0}}\cap \mathbb{S}^{m-1}\right) .
	\end{equation*}%
	Then, by the evident compactness argument with respect to the sphere $%
	\mathbb{S}^{m-1}$ already alluded to in the beginning of the proof of
	Theorem  \ref{domin}, we derive the existence of a neighborhood $\tilde{U%
	}$ of $x_{0}$ and of some $\alpha >0$ such that%
	\begin{equation*}
		\partial _{x}^{F}e(x,v)\subseteq \mathbb{B}_{\alpha }^{\ast }\left( 0\right)
		\quad \forall \left( x,v\right) \in \tilde{U}\times \mathbb{S}^{m-1}.
	\end{equation*}%
	In the case of (\ref{allnice}), the same relation (with $\alpha :=R$) is a
	direct consequence of Theorem \ref{domin} upon taking into account that $%
	C_{l}(x_{0})=X$ entails that $-C_{l}^{\ast }(x_{0})=\{0\}$. Now, the claimed
	statement on uniform Lipschitz continuity follows from \cite[Theorem 3.5.2]%
	{mordu}.
\end{proof}

\section{Subdifferential of the Gaussian probability function $\protect%
	\varphi \label{subdiffprob}$}
	In this section, we provide the required formulae for the Fr\'echet, the Mordukhovich, and the 
Clarke subdifferentials of the Gaussian probability function $\varphi$, 
defined in (\ref{probcons}). These results are next illustrated in Example \ref{illuex}, and in Subsection \ref{secc} to 
discuss the Lipschitz continuity and differentiability of $\protect\varphi $. Finally, we study in this section, 
Subsection  \ref{secc2},
the special and interesting setting of probability functions given by means of finite systems of smooth inequalities. 
In this case, formulae of the subdifferentials of $\varphi$ are expressed in terms of the initial
data in (\ref{probcons}), i.e., in terms of the function $g$. All this is done under our standard assumption (H).

\subsection{Main Result}

We start by recalling the following result on the interchange of Mordukhovich subdifferentials
and the integration sign when dealing with the integral functions of the form
$$
I_{f}(x):=\int\limits_{\omega \in \Omega }f(\omega ,x)d\mu.
$$
Here, $(\Omega ,\mathcal{A},\nu )$ a $\sigma $-finite complete measure space,
and $f:\Omega \times X\rightarrow \lbrack 0,+\infty ]$ is a normal integrand; that is, 
\par\noindent (i) $f$ is $\mathcal{A} 	\otimes \mathcal{B}(X)$-measurable,
 \par\noindent (ii) $f(\omega ,\cdot )$ is lsc for all $\omega \in \Omega $. 
 
 We assume that $I_{f}(x_{0})<+\infty $ for some  $x_{0}\in X$. Then we have the following result in which the 
 integral 
 $
 \int\limits_{\omega \in \Omega }\partial ^{M}f(\omega ,x_{0})d\nu
 $
 is to be understood in the Aumann's sense; that is, the set of Bochner integrals over all measurable selections of the
 multivalued mapping $\partial ^{M}f(\cdot ,x_{0})$ (see, e.g., \cite{valadier}).
 \begin{proposition}[\protect\cite{perez}]
	\label{keyresult} Assume that for some 
	$\delta >0$ and $\mathcal{K}\in L^{1}(\Omega ,\mathbb{R})$ we have
	\begin{equation}
	\partial _{x}^{F}f(\omega ,x)\subseteq \mathcal{K}(\omega )B_{1}^{*}(0)+C,\ \forall x\in
	B_\delta(x_0),\; \omega \in \Omega,  \label{dominance}
	\end{equation}%
	where $C\subseteq X^{\ast }$ is a closed convex cone with polar cone having a nonempty interior. 
	Then  
\begin{itemize}
\item[\rm(i)] $\partial^{M}I_{f}(x_{0})\subseteq \mathrm{cl}^*\,\left\{
		\int\limits_{\omega \in \Omega }\partial ^{M}f(\omega ,x_{0})d\nu \left(
		\omega \right) +C \right\}.$
\item[\rm(ii)]  Provided that $X$ is finite-dimensional, 
$$\partial^{M}I_{f}(x_{0})\subseteq
		\int\limits_{\omega \in \Omega }\partial ^{M}f(\omega ,x_{0})d\nu \left(
		\omega \right) +C.
			$$
\item[\rm(iii)]  $ \partial ^{\infty }I_{f}(x_{0})\subseteq C.$
\item[\rm(vi)]  $ \partial ^{C}I_{f}(x_{0})\subseteq\overline{\mathrm{co}}\,\left\{
		\int\limits_{\omega \in \Omega }\partial ^{M}f(\omega ,x_{0})d\nu \left(
		\omega \right) +C\right\} .$
\end{itemize}
 \end{proposition}
 
\bigskip\noindent Now, we are in a position to prove the main result of our
paper.

\begin{theorem}
	\label{main}Let $x_{0}\in X$ be such that $g(x_{0},0)<0$. Assume that the cone $%
	C_{l}(x_{0})$ has a non-empty interior for some  $l>0$. Then,%
\begin{itemize}
\item[\rm(i)] $
		\partial ^{M}\varphi (x_{0})\subseteq  \mathrm{cl}^*\,\left\{
		\int\limits_{v\in \mathbb{S}^{m-1}}\partial _{x}^{M}e(x_{0},v)d\mu
		_{\zeta}(v)-C_{l}^{\ast }(x_{0})\right\} $
\item[\rm(ii)]  Provided that $X$ is finite-dimensional, 
$$\partial ^{M}\varphi (x_{0})\subseteq  
		\int\limits_{v\in \mathbb{S}^{m-1}}\partial _{x}^{M}e(x_{0},v)d\mu
		_{\zeta}(v)-C_{l}^{\ast }(x_{0}) .$$		
\item[\rm(iii)]  $\partial ^{\infty }\varphi (x_{0})\subseteq -C_{l}^{\ast }(x_{0}).$
\item[\rm(vi)]  $
		\partial ^{C}\varphi (x_{0})\subseteq\overline{\mathrm{co}}\,%
		\left\{ \int\limits_{v\in \mathbb{S}^{m-1}}\partial _{x}^{M}e(x_{0},v)d\mu
		_{\zeta }(v)-C_{l}^{\ast }(x_{0})\right\} .$
\end{itemize}
\end{theorem}

\begin{proof}
	We apply Proposition \ref{keyresult} by putting
	\begin{equation*}
		f\left( \omega ,x\right) :=e\left( x,\omega \right) ,\,C:=-C_{l}^{\ast
		}(x_{0}),
	\end{equation*}
	and using the measurable space  $(\mathbb{S}^{m-1},\mathcal{A},\mu _{\zeta })$,
with $\mathcal{A}$ being the $\sigma$-Algebra of measurable sets with respect to $\mu _{\zeta }$. 
It is known that $\mu _{\zeta }$ is $\sigma $-finite and complete. 
The measurability property of $f$ and
	the lower semicontinuity of $f(\omega ,\cdot )$ are consequences of the
	continuity of $e$ (see Proposition \ref{econt}).  The cone  $C^{*}=\overline{\mathrm{co}}\,%
	C_{l}(x_0)$ has a non-empty interior, by the current assumption. Condition (\ref{dominance}) is a
	consequence of Theorem \ref{domin} upon defining $\mathcal{K}(\omega ):=R$ for all 
	$\omega \in \Omega =\mathbb{S}^{m-1}$, and observing that $\mathcal{K}\in L^{1}(\mathbb{%
		S}^{m-1},\mathbb{R})$, due to $\mathbb{S}^{m-1}$ having finite ($\mu
	_{\zeta}$\,-) measure. Now, the claimed result follows from Proposition \ref%
	{keyresult} by taking into account that $I_{f}=\varphi $ thanks to (\ref%
	{fidef}).
\end{proof}

\smallskip

%

Our main result motivates some investigation about the impact of
the parameter $l>0$ in the definition of the cones $C_{l}^{\ast }(x_{0})$, $x_{0}\in X$. From Definition \ref%
{nicedirect}, it follows immediately that $(C_{l}(x_{0}))_{l\ge 0}$ forms a non-decreasing
family of closed cones, and hence%
\begin{equation}
C_{k}(x_{0})\subseteq C_{k+1}(x_{0});\quad C_{k}^{\ast }(x_{0})\supseteqq
C_{k+1}^{\ast }(x_{0})\quad \forall k\in \mathbb{N}.  \label{monotonefamily}
\end{equation}%
Moreover, $C_{k}(x_{0})$ having a non-empty interior as required in Theorem %
\ref{main}, implies that $C_{k+1}(x_{0})$ does so too. This means that the
upper estimates in the results of Theorem \ref{main} become increasingly
precise for $k\rightarrow \infty $. This immediately raises the question if
we may pass to the limit in this result. Let us then introduce the 
{\it{limiting cone of nice directions}}%
\begin{eqnarray*}
	&C_{\infty }(x_{0}):=\bigcup\limits_{k\in \mathbb{N}}C_{k}(x_{0})=& \\
	&\{h\in X\ |\ \exists k\in \mathbb{N:}\, g^{\circ}(\cdot ,z)(y;h)\leq 
	k \left\Vert z\right\Vert ^{-m}e^{\frac{\Vert z\Vert^2}{2\Vert L\Vert^2} }\left\Vert
	h\right\Vert, \forall y\in \mathbb{B}_{\frac{1}{k}}\left( x\right), \left\Vert z\right\Vert \geq k\}.&
\end{eqnarray*}%
The reader can simply notice (through Baire's Theorem) the non-emptiness of the interior of $%
C_{\infty }(x_{0})$ is equivalent to the non-emptiness of the interior of $C_{l}(x_0)$ for some 
$l>0$. As far as the singular subdifferential is concerned, we may
immediately pass to the limit:

\begin{proposition}
	\label{singsub}Fix $x_{0}\in X$ with $g(x_{0},0)<0$, and assume
	 that $C_{l}(x_{0})$ has a non-empty interior for  some $l>0$. Then $%
	\partial ^{\infty }\varphi (x_{0})\subseteq -C_{\infty }^{\ast }(x_{0})$.
\end{proposition}

\begin{proof}
	By Theorem \ref{main}(iii) we have that $\partial ^{\infty }\varphi
	(x_{0})\subseteq -C_{l}^{\ast }(x_{0})$. Since along with $C_{l} (x_{0})$
	the larger cones $C_{k}(x_{0})$ for $k\in \mathbb{N},k\geq l$, have non-empty
	interiors too, it follows that 
	\begin{equation*}
		\partial ^{\infty }\varphi (x_{0})\subseteq \bigcap_{k\in \mathbb{N},k\geq
			l}-C_{k}^{\ast }(x_{0})=-\left( \bigcup\limits_{k\in \mathbb{N}}C_{k}(x_{0})\right) ^{\ast
		}=-C_{\infty }^{\ast }(x_{0}),
	\end{equation*}%
	where the first equality relies on (\ref{monotonefamily}).
\end{proof}

\bigskip\noindent In order to formulate a corresponding result for the
Mordukhovich and Clarke subdifferentials, we need an additional boundedness
assumption:

\begin{proposition}
	Fix $x_{0}\in X$ with  $g(x_{0},0)<0$, and assume 
	 that $C_{l}(x_{0})$ has a non-empty interior for some $l>0$. Moreover, suppose that $%
	\partial _{x}^{M}e(x_{0},v)$ is integrably bounded; i.e., there exists some
	 function $\mathcal{R}:\mathbb{S}^{m-1}\rightarrow \mathbb{R}_+$ with $\int\limits_{\mathbb{S}^{m-1}}\mathcal{R}(v)d\mu _{\zeta }(v)<\infty$ 
	 such that 
	\begin{equation*}
		\partial
		_{x}^{M}e(x_{0},v)\subseteq \mathbb{B}_{\mathcal{R}(v)}^{\ast }(0)\quad \mu
		_{\zeta }-a.e.\ v\in \mathbb{S}^{m-1}.
	\end{equation*}%
	Then
	\begin{equation*}
		\partial ^{M}\varphi (x_{0})\subseteq \partial ^{C}\varphi (x_{0})\subseteq\mathrm{cl%
		}\left\{ \int\limits_{v\in \mathbb{S}^{m-1}}\partial
		_{x}^{M}e(x_{0},v)d\mu _{\zeta }(v)\right\} -C_{\infty }^{\ast }(x_{0}).
	\end{equation*}
\end{proposition}

\begin{proof}
	For the purpose of abbreviation, put%
	\begin{equation*}
		\mathcal{I}:=\int\limits_{v\in \mathbb{S}^{m-1}}\partial
		_{x}^{M}e(x_{0},v)d\mu _{\zeta }(v).
	\end{equation*}%
	From our assumption on $\partial _{x}^{M}e(x_{0},v)$, being integrably
	bounded, it follows that $\mathcal{I}$ is bounded too. Consequently, $\mathrm{cl}%
	^{\ast }\mathcal{I}$ is $w^{\ast }$-compact. With $C_{l}(x_{0})$ having
	a non-empty interior, for all $k\in \mathbb{N}$ with $k\geq l$, from 
	Theorem \ref{main}(i)  it follows that%
	\begin{eqnarray*}
		\partial ^{M}\varphi (x_{0}) &\subseteq &\mathrm{cl}%
		^{\ast }\left\{ \mathcal{I}-C_{k}^{\ast }(x_{0})\right\} 
		=\mathrm{cl}^{\ast }\mathcal{I}-C_{k}^{\ast
		}(x_{0})\quad \forall k\geq l.
	\end{eqnarray*}%
 Due to (\ref{monotonefamily}), we may continue as 
	\begin{equation}
	\partial ^{M}\varphi (x_{0})\subseteq \bigcap_{k\in \mathbb{N}}\left\{ 
	\mathrm{cl}^{\ast }\mathcal{I}-C_{k}^{\ast }(x_{0})\right\},
	\label{firstinclu}
	\end{equation}%
	which in turn, using again the $w^{\ast }$-compacity of  $\mathrm{cl}%
	^{\ast }\mathcal{I}$, gives us 
	\begin{equation*}
		\partial ^{M}\varphi (x_{0})\subseteq \mathrm{cl}^{\ast }\mathcal{I}%
		-\bigcap_{k\in \mathbb{N}}C_{k}^{\ast }(x_{0})=\mathrm{cl}^{\ast }\mathcal{I}%
		-\left( \bigcup\limits_{k\in \mathbb{N}}C_{k}(x_{0})\right) ^{\ast }=\mathrm{%
			cl}^{\ast }\mathcal{I}-C_{\infty }^{\ast }(x_{0}).
	\end{equation*}%
	Now, by \cite[Theorem 3.57]{mordu}, by Proposition \ref{singsub}, and by convexity of $%
	C_{\infty }^{\ast }(x_{0})$, we arrive at 
	\begin{eqnarray*}
		\partial ^{C}\varphi (x_{0}) &=&\overline{\mathrm{co}}\,\left\{
		\partial ^{M}\varphi (x_{0})+\partial ^{\infty }\varphi (x_{0})\right\}\\
		&\subseteq&\overline{\mathrm{co}}\,\left\{ \mathrm{cl}^{\ast }%
		\mathcal{I}-C_{\infty }^{\ast }(x_{0})-C_{\infty }^{\ast }(x_{0})\right\} \\
		&=&\overline{\mathrm{co}}\,\left\{ \mathrm{cl}^{\ast }%
		\mathcal{I}-C_{\infty }^{\ast }(x_{0})\right\}.
	\end{eqnarray*}%
	Now, as a consequence of \cite[Theorem 3.1]{pucci}, the strong closure $\mathrm{cl\,}%
	\mathcal{I}$ is convex (the measure $\mu_\zeta$ being nonatomic), so that 
	$\mathrm{cl}^{\ast }\mathcal{I}=\mathrm{cl\,}%
	\mathcal{I}$ is convex, and the last inclusion above reads
	\begin{equation*}
		\partial ^{C}\varphi (x_{0})\subseteq\mathrm{cl}\mathcal{I}%
		-C_{\infty }^{\ast }(x_{0}).
	\end{equation*}%
	This finishes the proof of our proposition.
\end{proof}

\subsection{An illustrating example}

In the following, we provide an example which, on the one hand,
serves as an illustration of our main result Theorem \ref{main} and, on the
other hand, shows that even for a continuously differentiable inequality $%
g\left( x,\xi \right) \leq 0$, satisfying a basic constraint qualification,
the associated probability function $\varphi $ may fail to be
differentiable, actually even to be locally Lipschitzian (though it is
continuous due to the constraint qualification).

\begin{example}
	\label{illuex}Define the function $g:\mathbb{R}\times \mathbb{R}%
	^{2}\rightarrow \mathbb{R}$ by%
	\begin{equation*}
		g\left( x,z_{1},z_{2}\right) :=\alpha (x)e^{h\left( z_{1}\right) }+z_{2}-1,
	\end{equation*}%
	where%
	\begin{equation*}
		\alpha (x):=\left\{ 
		\begin{tabular}{ll}
			$x^{2}$ & $x\geq 0$ \\ 
			$0$ & $x<0,$%
		\end{tabular}%
		\right. 
		\end{equation*}
		\begin{equation*}
		 h\left( t\right) :=-1-4\log \left( 1-\Phi (t)\right) ;\quad
		\Phi (t):=\frac{1}{\sqrt{2\pi }}\int\limits_{-\infty }^{t}e^{-\tau
			^{2}/2}d\tau ,
	\end{equation*}%
	i.e., $\Phi $ is the distribution function of the one-dimensional standard
	normal distribution. Moreover, let $\xi $ have a bivariate standard normal
	distribution, i.e., 
	\begin{equation*}
		\xi =\left( \xi _{1},\xi _{2}\right) \sim \mathcal{N}\left( \left(
		0,0\right) ,\left( 
		\begin{array}{cc}
			1 & 0 \\ 
			0 & 1%
		\end{array}%
		\right) \right) .
	\end{equation*}%
	The following properties are shown in the Appendix:
	
	\begin{enumerate}
		\item $g$ is continuously differentiable.
		
		\item $g$ is convex in $\left( z_{1},z_{2}\right) $.
		
		\item $g\left( 0,0,0\right) <0$.
		
		\item $C_{1}(0)=\left( -\infty ,0\right] $.
		
		\item $\int\limits_{v\in \mathbb{S}^{1}}\partial _{x}^{M}e(0,v)d\mu
		_{\zeta }(v)\subseteq \left( -\infty ,0\right] $.
		
		\item $\varphi $ fails to be locally Lipschitzian in $0$.
	\end{enumerate}
	
	Observe that, by 1. and 2., $g$ satisfies our basic data
	assumptions, (H), and that 3. forces the probability function $\varphi $ to be
	continuous. 
	On the other hand, by 6., $\varphi $
	is not locally Lipschitzian -much less differentiable - in $0$ despite the
	continuous differentiability of $g$ and the satisfaction of Slater's
	condition. Now, Theorem \ref{main}(ii), along with 4. and 5. provides that 
	\begin{equation*}
		\partial ^{M}\varphi (0)\subseteq \left( -\infty ,0\right] -[0,\infty
		)=\left( -\infty ,0\right] ,\newline
		\quad \partial ^{\infty }\varphi (0)\subseteq \left( -\infty ,0\right] .
	\end{equation*}%
	On the other hand, analytical verification along with the formula for $%
	\varphi $ provided in the Appendix
	(or alternatively visual inspection of the graph of $\varphi $) yields that $\partial ^{M}\varphi
	(0)=\{0\}$ and $\partial ^{\infty }\varphi (0)=\left( -\infty ,0\right] $,
	so that the upper estimate for the singular subdifferential is strict, while
	the one for the basic subdifferential is not (nevertheless this upper
	estimate is nontrivial due to being smaller than the whole space).
\end{example}

\subsection{Lipschitz continuity and differentiability of $\protect\varphi $}\label{secc}

The following result on Lipschitz continuity of the probability
function $\varphi $ is an immediate consequence of Clarke's Theorem on
interchanging subdifferentiation and integration \cite[Theorem 2.7.2]{clarke}
and of Corollary \ref{uniformlipschitz}:

\begin{theorem}
	\label{lipprobfunc}Fix $x\in X$ such that $g(x,0)<0$. Under one of
	the alternative conditions {\rm(\ref{bounded})} or {\rm(\ref{allnice})}, the
	probability function $\varphi $ is locally Lipschitz near $x$ and the
	following estimate holds true{\rm:}%
	\begin{equation}
	\partial ^{C}\varphi (x)\subseteq \int\limits_{v\in \mathbb{S}%
		^{m-1}}\partial _{x}^{C}e(x,v)d\mu _{\zeta }(v).  \label{lipinclu}
	\end{equation}
\end{theorem}

\quad

\noindent The next result provides conditions for differentiability of the
probability function $\varphi $; recall that $\# A$ denotes the cardinal of a set $A$.

\begin{proposition}
	\label{strictdiff} In addition to the assumptions of Theorem {\rm\ref%
	{lipprobfunc}}, assume that%
	\begin{equation}
	\#\partial _{x}^{C}e(x,v)=1\quad \mu _{\zeta }\text{-a.e.}\,v\in \mathbb{S}%
	^{m-1}.  \label{singlevalued}
	\end{equation}%
	Then $\varphi $ is strictly differentiable at $x$ and 
	\begin{equation*}
		\nabla \varphi (x)=\int\limits_{v\in \mathbb{S}^{m-1}}\nabla
		_{x}e(x,v)d\mu _{\zeta }(v).
	\end{equation*}%
	Consequently, if $X$ is finite-dimensional and {\rm(\ref{singlevalued})} holds true 
	in some neighborhood of $x$, then $\varphi $ is even continuously
	differentiable at $x$.
\end{proposition}

\begin{proof}
	Assumption (\ref{singlevalued}) entails that the integral in (\ref{lipinclu}%
	) reduces to a singleton. On the other hand, the subdifferential on the
	left-hand side of (\ref{lipinclu}) is nonempty, since $\varphi $ is locally
	Lipschitz near $x$ (see \cite[Proposition 2.1.2]{clarke}). Hence, the
	inclusion (\ref{lipinclu}) yields the single-valuedness of $\partial
	^{C}\varphi (x)$ as well as the equality 
	\begin{equation*}
		\partial ^{C}\varphi (x)= \int\limits_{v\in \mathbb{S}%
			^{m-1}}\partial _{x}^{C}e(x,v)d\mu _{\zeta }(v).
	\end{equation*}%
	Now, since a locally Lipschitzian function reducing to a singleton at some
	point is strictly differential at this point with gradient equal to the
	(single-valued) subdifferential (see \cite[Proposition 2.2.4]{clarke}), it follows
	that $\varphi $ is strictly differentiable at $x_{0}$ and $\partial
	^{C}\varphi (x_{0})=\left\{ \nabla \varphi (x_{0})\right\} $. Likewise, the
	local Lipschitz continuity of $e(\cdot ,v)$ around $x_{0}$ for all $v\in 
	\mathbb{S}^{m-1}$ (see Corollary \ref{uniformlipschitz}) yields along with (\ref%
	{singlevalued}) that 
	\begin{equation*}
		\partial _{x}^{C}e(x_{0},v)=\{\nabla _{x}e(x_{0},v)\}\quad \mu _{\zeta
		}-a.e.\,v\in \mathbb{S}^{m-1}.
	\end{equation*}%
	Altogether, we have shown the first assertion of our Proposition. The second
	assertion on continuous differentiability follows from \cite[Corollary to Prop.
	2.2.4]{clarke}.
\end{proof}

\subsection{Application to a finite system of smooth inequalities}\label{secc2}

In order to benefit from Theorem \ref{main}, one has to be able to
express the integrand $\partial _{x}^{M}e(x_{0},v)$ in terms of the initial
data in (\ref{probcons}), i.e., in terms of the function $g$. We will
illustrate this for the case of a probability function defined over a finite
system of continuously differentiable inequalities which are convex in their
second argument:%
\begin{equation}
\varphi (x):=\mathbb{P}\left( g_{i}\left( x,\xi \right) \leq 0,\ 
i=1,\ldots ,p\right),\  x\in X.  \label{probfunc2}
\end{equation}%
Clearly, this can be recast in the form of (\ref{probcons}) upon defining 
\begin{equation}
g:=\max\limits_{i=1,\ldots ,p}g_{i},  \label{maxfunc}
\end{equation}%
where $g$ is locally Lipschitz as required and convex in the second argument
because the $g_{i}$'s are supposed to be so. Since $g\left( x,0\right) <0$
implies that $g_{i}\left( x,0\right) <0$ for all $i=1,\ldots ,p$, we may
associate with each component a function $\rho _{i}$ satisfying the relation 
$g_{i}\left( x,\rho _{i}\left( x,v\right) Lv\right) =0$, as we did in (\ref%
{rhodef}). The relation between $\rho $ associated via (\ref{rhodef}) with $%
g $ in (\ref{maxfunc}) is, clearly,
\begin{equation}
\rho \left( x,v\right) =\min_{i=1,\ldots ,p}\rho _{i}\left( x,v\right) \quad
\forall x:g\left( x,0\right) <0,\,\,\forall v\in F(x).  \label{minfunc}
\end{equation}%
Note, however, that unlike $\rho $, the functions $\rho _{i}$ are
continuously differentiable because the $g_{i}$'s are so. This is a
consequence of the Implicit Function Theorem (see \cite[Lemma 3.2]{ackooij}),
which moreover yields  the gradient formulae, for all $x$ with $g\left( x,0\right) <0$ and all $v\in F(x)$,%
\begin{eqnarray*}
\nabla _{x}\rho _{i}\left( x,v\right) =-\frac{1}{\left\langle \nabla
	_{z}g_{i}\left( x,\rho \left( x,v\right) Lv\right) ,Lv\right\rangle }%
\nabla _{x}g_{i}\left( x,\rho \left( x,v\right) Lv\right)
\label{gradmin},\ i=1,\ldots ,p. 
\end{eqnarray*}%
In the following proposition, we provide an explicit upper estimate of the subdifferential set $%
\partial _{x}^{M}e(x_{0},v)$ in terms of the initial data, which can be used
in the formula of Theorem \ref{main}  to get an upper estimate for
the subdifferential of the probability function (\ref{probfunc2}):

\begin{proposition}
	\label{mordexplic}Fix $x\in X$ such that $g_{i}\left( x,0\right) <0$
	for $i=1,\ldots ,p$. Then, for every $l>0$, there exists some $R>0$ such that
	the radial probability function associated with $g$ in {\rm(\ref{maxfunc})} via {\rm(\ref{edef})}
	satisfies%
	\begin{equation*}
		\partial _{x}^{M}e(x,v)\subseteq \left\{ 
		\begin{array}{ll}
			-\bigcup\limits_{i\in T(v)}\left\{ \frac{\chi \left( \rho \left(
				x,v\right) \right) }{\left\langle \nabla _{z}g_{i}\left( x,\rho \left(
				x,v\right) Lv\right) ,Lv\right\rangle }\nabla _{x}g_{i}\left( x,\rho
			\left( x,v\right) Lv\right) \right\} & v\in F\left( x\right) \\ 
			&\\
			\mathbb{B}_{R}^{\ast }\left( 0\right) -C_{l}^{\ast }(x) & v\in I\left(
			x\right).%
		\end{array}%
		\right.
	\end{equation*}%
	Here, $T(v):=\left\{ i\in \left\{ 1,\ldots ,p\right\} \ |\ \rho _{i}\left(
	x,v\right) =\rho \left( x,v\right) \right\} $.
\end{proposition}

\begin{proof}
	Fix an arbitrary $v\in \mathbb{S}^{m-1}$. Given the continuity of $e$, we
	exploit the following representation \cite[Theorem 2.34]{mordu} of the
	Mordukhovich subdifferential in terms of the Fr\'{e}chet subdifferential, which holds true
	in Asplund spaces (hence, in particular for reflexive Banach spaces) 
	\begin{align*}
		x^{\ast }\in \partial _{x}^{M}e(x,v)&\Longleftrightarrow \exists\,
		x_{n}\rightarrow _{n}x \text{ and }\exists\, x_{n}^{\ast }\rightharpoonup _{n}x^{\ast
		}: \ x_{n}^{\ast }\in \partial _{x}^{F}e(x_{n},v).
	\end{align*}%
	Then, the inclusion $\partial _{x}^{M}e(x,v)\subseteq \mathbb{B}%
	_{R}^{\ast }\left( 0\right) -C_{l}^{\ast }(x)$ follows from Theorem %
	\ref{domin}, since  $\mathbb{B}_{R}^{\ast }\left( 0\right)$ is weak*-compact and 
	$C_{l}^{\ast }(x)$ is weak*-closed, entailing that $ \mathbb{B}%
	_{R}^{\ast }\left( 0\right) -C_{l}^{\ast }(x)$  is weak*-closed. 
	This yields the desired estimate of $\partial _{x}^{M}e(x,v)$ when $v\in I \left(x\right)$.  
	
	Suppose now in addition that $v\in F\left(
	x\right) $, and, according to Lemma \ref{opendom}, let $U$ be a neighborhood of $x$ such that,
	for all $y\in U$,
	\begin{equation*}
		g\left( y,0\right) <0,\,\,v\in F(y).
	\end{equation*}%
	From the proof of Theorem \ref{chainrule} we have seen that%
	\begin{equation*}
		\partial _{x}^{F}e(y,v)=\chi \left( \rho (y,v)\right)  \partial
		_{x}^{F}\rho (y,v),\quad \forall y\in U,
	\end{equation*}%
	which, by continuity of $\chi $ and by 1. in Lemma \ref{rholem}, immediately
	entails that 
	\begin{equation*}
		\partial _{x}^{M}e(x,v)=\chi \left( \rho (x,v)\right)  \partial
		_{x}^{M}\rho (x,v).
	\end{equation*}%
	From (\ref{minfunc}) and the calculus rule for minimum functions \cite[Proposition
	1.113]{mordu} we conclude that%
	\begin{equation*}
		\partial _{x}^{M}\rho (x,v)\subseteq \bigcup\limits_{i\in T(v)}\nabla
		_{x}\rho _{i}(x,v).
	\end{equation*}%
	with $T(v)$ being defined as in the statement of the Proposition. Now, the assertion
	follows from (\ref{gradmin}).
\end{proof}

\bigskip\noindent We provide next a concrete characterization for the local
Lipschitz continuity/differentiability of the probability function $\varphi$, defined in  (\ref{probfunc2}), 
along with an explicit subdifferential/gradient formula:

\begin{theorem}
	\label{clarkeappli}Fix $x_{0}\in X$ with $g\left( x_{0},0\right) <0$,
	and assume that for some $l>0$ it holds, for $i=1,\ldots ,p$,
	\begin{equation}
	\left\Vert \nabla _{x}g_{i}(x,z)\right\Vert \leq l \left\Vert
	z\right\Vert ^{-m}e^{\frac{\left\Vert z\right\Vert^2}{2\Vert L \Vert^2} }\quad \forall x\in 
	\mathbb{B}_{1/l}\left( x_{0}\right), \,\left\Vert z\right\Vert
	\geq l.  \label{growth}
	\end{equation}%
	Then the probability function {\rm(\ref{probfunc2})} is locally Lipschitz near $%
	x_{0}$ and there exists a nonnegative number $R\leq \sup\{ \| x^*\| \mid x^*\in    \partial^M_x e(x_0,v) \text{ and } v\in I(x_0)\}$ such that
	\begin{eqnarray*}
		\partial ^{C}\varphi (x_{0})&\subseteq& -\int\limits_{v\in F(x_{0})}\mathrm{%
			co\,}\left\{ \bigcup_{i\in T(v)}\frac{\chi \left( \rho \left( x_{0},v\right) \right) 
			\nabla _{x}g_{i}\left( x_0,\rho \left( x_{0},v\right)Lv\right) }{%
			\left\langle \nabla _{z}g_{i}\left( x_0,\rho \left( x_{0},v\right) Lv\right)
			,Lv\right\rangle }\right\} d\mu _{\zeta }(v) \\
			&&\\
		&&+\mu _{\zeta }(I(x_{0})) \mathbb{B}^{*}_{R}\left( 0\right) .
	\end{eqnarray*}
\end{theorem}

\begin{proof}
	As a maximum of finitely many smooth functions, $g$ is Clarke-regular, so
	that Clarke's directional derivative of $g$ coincides with its usual
	directional derivative. Hence, by Danskin's Theorem and by (\ref{growth}),
	we get the following estimate, for all $h\in X$, $x\in \mathbb{B}_{1/l}\left( x_{0}\right)$
	and $\left\Vert z\right\Vert \geq l$,
	\begin{eqnarray*}
		g^{\circ}(\cdot ,z)(x;h) &=&\left\langle \nabla _{x}g(x,z),h\right\rangle \\
		&=&\max
		\left\{ \left\langle \nabla _{x}g_{i}(x,z),h\right\rangle
		g_{i}(x,z)=g(x,z)\right\} \\
		&\leq &\max_{i=1,\ldots ,p}\left\langle \nabla _{x}g_{i}(x,z),h\right\rangle
		\leq l \left\Vert z\right\Vert ^{-m}e^{\frac{\left\Vert z\right\Vert^2}{2\Vert L\Vert^2}}\left\Vert h\right\Vert.
	\end{eqnarray*}%
	Hence, $C_{l}(x_{0})=X$ and, so,  Theorem \ref{lipprobfunc} guarantees that $%
	\varphi $ in (\ref{probfunc2}) is locally Lipschitz near $x_{0}$ and that 
	\begin{equation}
	\partial ^{C}\varphi (x_{0})\subseteq \int\limits_{v\in F(x_{0})}\partial
	_{x}^{C}e(x_{0},v)d\mu _{\zeta }(v)+\int\limits_{v\in I(x_{0})}\partial
	_{x}^{C}e(x_{0},v)d\mu _{\zeta }(v).  \label{sumclarke}
	\end{equation}%
	Since $e\left( \cdot ,v\right) $ is locally Lipschitzian for all $v\in 
	\mathbb{S}^{m-1}$, it follows from \cite[Theorem 3.57]{mordu} and from
	Proposition \ref{mordexplic} that 
	\begin{eqnarray*}
		\partial _{x}^{C}e(x_{0},v) &=&\overline{\mathrm{co}}\,\left\{\partial
		_{x}^{M}e(x_{0},v)\right\} \\
		&=&-\mathrm{co\,}\left\{ \bigcup_{i\in T(v)} \frac{\chi \left( \rho \left( x_{0},v\right)
			\right)\nabla _{x}g_{i}\left( x_0,\rho \left(
		x_{0},v\right) Lv\right) }{\left\langle \nabla _{z}g_{i}\left( x_0,\rho \left( x_{0},v\right)
			Lv\right) ,Lv\right\rangle } \right\}.
	\end{eqnarray*}%
	Hence, the first term on the right-hand side of (\ref{sumclarke}) coincides
	with the integral term in the asserted formula above. As for the second
	term, observe that $\partial _{x}^{C}e(x_{0},v)\subseteq \mathbb{B}^*%
	_{R}\left( 0\right) $ for some $R>0$ by Theorem \ref{domin}, which yields
	the second term in the upper estimate of this theorem.
\end{proof}

\bigskip

\noindent From Theorem \ref{clarkeappli} and Proposition \ref{strictdiff},
we immediately derive the following:

\begin{corollary}
	If in the setting of Theorem {\rm\ref{clarkeappli}} one has that $\mu _{\zeta
	}(I(x_{0}))=0$ {\rm(}in particular, under assumption {\rm(\ref{bounded})}{\rm)}, or the constant $R$ in Theorem \ref{clarkeappli} is zero, then
	\begin{equation*}
		\partial ^{C}\varphi (x_{0})\subseteq -\int\limits_{v\in \mathbb{S}^{m-1}}%
		\mathrm{co\,}\left\{ \bigcup_{i\in T(v)}\frac{\chi \left( \rho \left( x_{0},v\right) \right) \nabla _{x}g_{i}\left( x_0,\rho \left( x_{0},v\right)
		Lv\right)}{%
			\left\langle \nabla _{z}g_{i}\left( x_0,\rho \left( x_{0},v\right) Lv\right)
			,Lv\right\rangle } \right\} d\mu _{\zeta }(v).
	\end{equation*}
	If, in addition, for $\mu _{\zeta }$-a.e.$\,v\in \mathbb{S}^{m-1}$ we have
	that $\#T(v)=1$ (say: $T(v)=\{i^{\ast }(v)\}$), then the probability
	function {\rm(\ref{probfunc2})} is strictly differentiable with gradient%
	\begin{equation*}
		\nabla \varphi (x_{0})=-\int\limits_{v\in \mathbb{S}^{m-1}}\mathrm{\,}\frac{%
			\chi \left( \rho \left( x_{0},v\right) \right)\nabla _{x}g_{i^{\ast }(v)}\left( x_0,\rho \left(
		x_{0},v\right) Lv\right) }{\left\langle \nabla
			_{z}g_{i^{\ast }(v)}\left( x_0,\rho \left( x_{0},v\right) Lv\right)
			,Lv\right\rangle } d\mu _{\zeta }(v).
	\end{equation*}
\end{corollary}

\begin{remark}
	It is worth mentioning that  under the strengthened (compared with (\ref{growth})) growth condition%
	\begin{equation*}
		\exists l>0:\left\Vert \nabla _{x}g_{i}(x,z)\right\Vert \leq l
		e^{\left\Vert z\right\Vert }\quad \forall x\in \mathbb{B}_{1/l}\left(
		x_{0}\right), \,\,\left\Vert z\right\Vert \geq l,\,\,
		i=1,\ldots ,p
	\end{equation*}%
	the constant $R$ in Theorem \ref{clarkeappli} and Corollary above is zero,  as it can be seen in  (\ref{ConstantR}) (see also \cite[Theorem 3.6 and Theorem 4.1]{ackooij2}). 
\end{remark}

\section{Appendix}

We verify in this Appendix properties 1.-6. in Example \ref{illuex}.

The continuous differentiability of $g$ stated in 1. is obvious from the
corresponding property of $\alpha $ and $h$. For $h$, this relies on the
smoothness of the distribution function of the one-dimensional standard
normal distribution $\Phi $ and on the fact that the argument $1-\Phi (t)$
of the logarithm is always strictly positive. 

By nonnegativity of $\alpha $
it is sufficient to check that $e^{h\left( t\right) }$ is convex in order to
verify 2. To do so, it is sufficient to show that $h$ itself is convex,
which by definition would follow from the concavity of $\log \left( 1-\Phi
(t)\right) $. This, however, is a consequence of $\log \Phi $ being concave,
which in turn implies that $\log \left( 1-\Phi \right) $ is concave (see 
\cite[Theorem 4.2.4]{prek}). 

Statement 3. follows immediately from the
definition of the functions.

As for 4., observe first that, by continuous differentiability of $g$, 
\begin{equation*}
	g^{\circ}(\cdot ,z)(x;-1)=\nabla _{x}g\left( x,z_{1},z_{2}\right) \cdot \left(
	-1\right) =-\alpha ^{\prime }(x)e^{h\left( z_{1}\right) }\leq 0\quad \forall
	x,z_{1},z_{2}\in \mathbb{R},
\end{equation*}%
whence $-1\in C_{1}(0)$ by Definition \ref{nicedirect}. On the other hand,
putting $x:=1$ and $z:=(1,0)$, we have that $x\in \mathbb{B}_{1}\left(
0\right) $, $\left\Vert z\right\Vert =1$ and%
\begin{equation*}
	g^{\circ}(\cdot ,z)(x;1)=\nabla _{x}g\left( 1,1,0\right) \cdot 1=\alpha
	^{\prime }(1)e^{h\left( 1\right) }=2e^{h\left( 1\right) }\approx 1161,
\end{equation*}%
whereas, due to $m=2$ in this example,%
\begin{equation*}
	\left\Vert z\right\Vert ^{-m}e^{\frac{\left\Vert z\right\Vert^2}{2\Vert L\Vert^2} }=\sqrt{e}%
	\approx 1.65.
\end{equation*}%
Therefore, by Definition \ref{nicedirect}, $1\notin C_{1}(0)$. Since $%
C_{1}(0)$ is a closed cone, this together with $-1\in C_{1}(0)$ yields $%
C_{1}(0)=\left( -\infty ,0\right] $.

For proving 5., it is sufficient to show that%
\begin{equation}
\partial _{x}^{M}e(0,v)\subseteq \left( -\infty ,0\right] \quad \forall v\in 
\mathbb{S}^{1}.  \label{inclumordu}
\end{equation}%
In order to calculate $\partial _{x}^{M}e(0,v)$ for an arbitrarily fixed $%
v\in \mathbb{S}^{1}$, we have to compute first the partial Fr\'{e}chet
subdifferentials $\partial _{x}^{F}e(x,v)$ for $x$ in a neighborhood $U$ of 
$0$. Define $U$ such that $g(x,0,0)<0$ for all $x\in U$ (as a consequence of
the already shown relation $g(0,0,0)<0$). If $x<0$, then, by definition of $%
e $ and $g$, 
\begin{equation*}
	e(x,v)=\mu _{\eta }\left( \left\{ r\geq 0\ |\ g\left( x,rLv\right) \leq
	0\right\} \right) =\mu _{\eta }\left( \left\{ r\geq 0\ |\ rLv_{2}\leq 1\right\}
	\right) .
\end{equation*}%
Hence, for $x<0$, $e(x,v)$ does not depend on its first argument locally
around $x$. Therefore, $\partial _{x}^{F}e(x,v)=\{0\}$ for all $x<0$. Now,
consider some $x\in U$ with $x\geq 0$ and $x^{\ast }\in \partial
_{x}^{F}e(x,v)$. If $v\in I(x)$, then $\partial _{x}^{F}e(x,v)\subseteq
\left\{ 0\right\} $ (see Corollary \ref{loclipe}(ii)). If, in contrast, $v\in F(x)$,
then, by Theorem \ref{chainrule} (putting $w:=\pm 1$ there and observing
that, by continuous differentiability of $g$, the partial Clarke
subdifferentials reduce to partial gradients), 
\begin{eqnarray*}
	x^{\ast } &=&\frac{-\chi \left( \rho (x,v)\right) \nabla _{x}g\left(
	x,\rho (x,v)Lv\right) }{\left\langle \nabla
		_{z}g\left( x,\rho (x,v)Lv\right) ,Lv\right\rangle }=\frac{-2xe^{h\left( \rho (x,v)v_{1}\right) }\chi
		\left( \rho (x,v)\right) }{\left\langle \nabla _{z}g\left( x,\rho
		(x,v)Lv\right) ,Lv\right\rangle }\leq 0. 
\end{eqnarray*}%
Here, the inequality relies on $x\geq 0$, on $\chi $ being positive as a
density and on 
\begin{equation*}
	\left\langle \nabla _{z}g\left( x,\rho (x,v)Lv\right) ,Lv\right\rangle \geq 
	\frac{-g(x,0,0)}{\rho (x,v)}>0
\end{equation*}%
by Lemma \ref{lowest}. Altogether, we have shown that $\partial _{x}^{F}e(x,v)\subseteq
\left( -\infty ,0\right] $ for all $x\in U$. This entails that also $%
\partial _{x}^{M}e(x,0)\subseteq \left( -\infty ,0\right] $. Since $v\in 
\mathbb{S}^{1}$ has been fixed arbitrarily, the desired relation (\ref%
{inclumordu}) follows.

In order to show 6. we provide first a formula for the probability function $%
\varphi $. If $t\leq 0$, then, by definition of $g$, 
\begin{equation*}
	\varphi (t)=\mathbb{P}\left( g\left( x,\xi _{1},\xi _{2}\right) \leq
	0\right) =\mathbb{P}\left( \xi _{2}\leq 1\right) =\Phi (1)
\end{equation*}%
because $\xi _{2}\sim \mathcal{N}\left( 0,1\right) $ by the distribution
assumption on $\xi $ in Example \ref{illuex}. If $t>0$, then, again by the
assumed distribution of $\xi $, 
\begin{eqnarray*}
	\varphi (t) &=&\mathbb{P}\left( \xi _{2}\leq 1-t^{2}e^{h\left( \xi
		_{1}\right) }\right) =\frac{1}{2\pi }\int\limits_{-\infty }^{\infty }\left(
	\int\limits_{-\infty }^{1-t^{2}e^{h\left( z_{1}\right) }}e^{-\left(
		z_{1}^{2}+z_{2}^{2}\right) /2}dz_{2}\right) dz_{1} \\
	&=&\frac{1}{\sqrt{2\pi }}\int\limits_{-\infty }^{\infty
	}e^{-z_{1}^{2}/2}\cdot \frac{1}{\sqrt{2\pi }}\left( \int\limits_{-\infty
	}^{1-t^{2}e^{h\left( z_{1}\right) }}e^{-z_{2}^{2}/2}dz_{2}\right) dz_{1} \\
	&=&\frac{1}{\sqrt{2\pi }}\int\limits_{-\infty }^{\infty }e^{-s^{2}/2}\cdot
	\Phi \left( 1-t^{2}e^{h\left( s\right) }\right) ds.
\end{eqnarray*}%
Now, we are going to show that $\varphi $ fails to be locally Lipschitz
around $0$. Observe first that, since $ \Phi$ is increasing as a distribution
function, $h$ is increasing too by its definition. Then, for any $s,t$
satisfying $s\geq \Phi ^{-1}\left( 1-\sqrt{t}\right) $ (recall that $\Phi $
is strictly increasing and so its inverse exists) it holds that%
\begin{equation*}
	h(s)\geq h\left( \Phi ^{-1}\left( 1-\sqrt{t}\right) \right) =-1-\log t^{2}.
\end{equation*}%
Therefore, $t^{2}e^{h(s)}\geq e^{-1}$. Thus, we have shown that%
\begin{equation*}
	\Phi \left( 1\right) -\Phi \left( 1-t^{2}e^{h(s)}\right) \geq \Phi \left(
	1\right) -\Phi \left( 1-e^{-1}\right) =:\varepsilon \quad \forall s,t:s\geq
	\Phi ^{-1}\left( 1-\sqrt{t}\right) .
\end{equation*}%
With $\Phi $ being strictly increasing, we have that $\varepsilon >0$. Now,
for any $t>0$, we calculate%
\begin{eqnarray*}
	\varphi (0)-\varphi (t) &=&\Phi (1)-\frac{1}{\sqrt{2\pi }}%
	\int\limits_{-\infty }^{\infty }e^{-s^{2}/2}\cdot \Phi \left(
	1-t^{2}e^{h\left( s\right) }\right) ds \\
	&=&\frac{1}{\sqrt{2\pi }}\int\limits_{-\infty }^{\infty }e^{-s^{2}/2}\cdot
	\left( \Phi (1)-\Phi \left( 1-t^{2}e^{h\left( s\right) }\right) \right) ds \\
	&\geq &\varepsilon \frac{1}{\sqrt{2\pi }}\int\limits_{\Phi ^{-1}\left( 1-%
		\sqrt{t}\right) }^{\infty }e^{-s^{2}/2}ds=\varepsilon \left( 1-\frac{1}{%
		\sqrt{2\pi }}\int\limits_{-\infty }^{\Phi ^{-1}\left( 1-\sqrt{t}\right)
	}e^{-s^{2}/2}ds\right) \\
	&=&\varepsilon \left( 1-\Phi \left( \Phi ^{-1}\left( 1-\sqrt{t}\right)
	\right) \right) =\varepsilon \sqrt{t}.
\end{eqnarray*}%
Since $\varepsilon >0$, $\varphi $ fails to be locally Lipschitz around $0$,
which finally shows 6.

\end{document}